%% file: prime-index-asymptotics.tex
\documentclass[11pt, a4paper, reqno]{amsart}

\usepackage[T1]{fontenc}
\usepackage[utf8]{inputenc}

\usepackage{newtxtext}
\usepackage{newtxmath}

\usepackage[scaled=0.90]{helvet}

\usepackage{courier}

\usepackage[mathcal]{eucal}

\usepackage{mathrsfs}

\usepackage{xcolor}

\usepackage{microtype}                   
\usepackage[
    a4paper,
    top=2.5cm,
    bottom=2.5cm,
    left=2.5cm,
    right=2.5cm,
    headheight=14pt,
    headsep=1.2em,
    footskip=2.5em,
]{geometry}

\usepackage{fancyhdr}
\fancypagestyle{plain}{%
    \fancyhf{}
    
    \fancyfoot[C]{\small\thepage}
}

\usepackage{amsmath, amsthm}
\usepackage{mathtools}
\usepackage{bm}
\usepackage{nicefrac}

\usepackage{graphicx}
\usepackage{float}
\usepackage{booktabs}
\usepackage{caption}
\usepackage{tikz}
\usepackage{tikz-cd}       
\usetikzlibrary{arrows.meta, calc, positioning, decorations.pathmorphing,
                 decorations.pathreplacing,
                 shapes.geometric, patterns, shadows}

\usepackage[
    colorlinks  = true,
    linkcolor   = cyan!80!black,
    citecolor   = purple,
    urlcolor    = blue!60!black,
    backref     = page,
]{hyperref}
\usepackage[nameinlink, capitalize, noabbrev]{cleveref}

\definecolor{secrefcolor}{HTML}{00897B}
\newcommand{\secref}[1]{{\hypersetup{linkcolor=secrefcolor}\cref{#1}}}
\definecolor{figrefcolor}{HTML}{1E50C8}
\newcommand{\figref}[1]{{\hypersetup{linkcolor=figrefcolor}\cref{#1}}}

\usepackage{etoolbox}
\makeatletter
\begingroup
\catcode`\#=12
\gdef\thm@anchor@patch{%
  \patchcmd{\cref@thmoptarg}
    {\refstepcounter[#1]{#3}}
    {\refstepcounter[#1]{#3}\ifstrequal{#1}{#3}{}{\MakeLinkTarget[#1]{#3}}}
    {}{\PackageWarning{local}{theorem anchor patch failed}}}
\endgroup
\AtBeginDocument{\thm@anchor@patch}
\makeatother

\theoremstyle{plain}
\newtheorem{theorem}{Theorem}
\newtheorem{lemma}[theorem]{Lemma}
\newtheorem{proposition}[theorem]{Proposition}
\newtheorem{corollary}[theorem]{Corollary}

\theoremstyle{definition}

\theoremstyle{remark}

\newtheorem*{theorem*}{Theorem}
\newtheorem*{lemma*}{Lemma}
\newtheorem*{proposition*}{Proposition}
\newtheorem*{corollary*}{Corollary}
\newtheorem*{conjecture*}{Conjecture}
\theoremstyle{definition}
\newtheorem*{remark*}{Remark}

\makeatletter
\newcommand{\boldthmnote}{%
  \def\thmhead##1##2##3{%
    \thmname{##1}\thmnumber{\@ifnotempty{##1}{ }\@upn{##2}}%
    \thmnote{ {\bfseries(##3)}}}}
\makeatother

\AtBeginDocument{}

\makeatletter
\renewcommand{\@cite}[2]{{\small[{\textbf{#1}\if@tempswa , \textbf{#2}\fi}]}}
\makeatother

\definecolor{backrefcolor}{HTML}{8B4513}
\definecolor{citpagecolor}{HTML}{C00000}
\definecolor{geonavy}{RGB}{27,58,92}   
\definecolor{geogold}{RGB}{184,134,11}  

\makeatletter
\renewcommand*{\backref}[1]{}

\renewcommand*{\backrefalt}[4]{%
    \ifcase #1 %
    \or {\hspace{0.8em}\footnotesize #2}%
    \else {\hspace{0.8em}\footnotesize #2}%
    \fi
}

\makeatother

\definecolor{emailcolor}{HTML}{E6007E}
\makeatletter
\let\orig@email\email
\renewcommand{\email}[1]{\orig@email{\color{emailcolor}#1}}
\makeatother

\makeatletter
\renewcommand{\section}{\@startsection{section}{1}%
    {\z@}{-2.5ex \@plus -1ex \@minus -.2ex}{1.5ex \@plus .2ex}%
    {\normalfont\fontsize{16}{19}\bfseries}}
\renewcommand{\subsection}{\@startsection{subsection}{2}%
    {\z@}{-2ex \@plus -0.8ex \@minus -.2ex}{1ex \@plus .2ex}%
    {\normalfont\fontsize{13}{16}\bfseries}}
\renewcommand{\subsubsection}{\@startsection{subsubsection}{3}%
    {\z@}{-1.5ex \@plus -0.5ex \@minus -.1ex}{0.8ex \@plus .1ex}%
    {\normalfont\fontsize{11}{14}\bfseries\itshape}}

\renewcommand{\@settitle}{\begin{center}%
    \baselineskip20pt\relax
    \bfseries\fontsize{18}{22}\selectfont
    \@title
    \end{center}%
}
\let\uppercasenonmath\@gobble
\AtBeginDocument{}

\renewcommand{\@setauthors}{%
    \begingroup
    \trivlist
    \centering
    \fontsize{13}{16}\selectfont
    \@topsep30\p@\relax
    \advance\@topsep by -\baselineskip
    \item\relax
    \andify\authors
    \def\\{\protect\linebreak}%
    \authors
    \endtrivlist
    \endgroup
}
\makeatother

\definecolor{toccolor}{RGB}{27,58,92}

\makeatletter
\renewcommand{\tableofcontents}{%
    \par
    \begin{list}{}{%
        \leftmargin2.5pc \rightmargin2.5pc
        \listparindent\z@ \itemindent\z@
        \parsep\z@ \topsep6pt \partopsep\z@}
    \item\relax
    {\large\bfseries Contents}\par
    \vspace{4pt}%
    \@input{\jobname.toc}%
    \if@filesw
        \expandafter\newwrite\csname tf@toc\endcsname
        \immediate\openout \csname tf@toc\endcsname \jobname.toc\relax
    \fi
    \global\@nobreakfalse
    \end{list}%
}
\makeatother

\makeatletter
\renewcommand{\tocsection}[3]{%
    \indentlabel{\@ifnotempty{#2}{\color{toccolor}\ignorespaces#1 #2.\quad}}{\bfseries\color{toccolor}#3}}
\renewcommand{\tocsubsection}[3]{%
    \indentlabel{\@ifnotempty{#2}{\color{toccolor}\ignorespaces#1 #2.\quad}}{\color{toccolor}#3}}
\renewcommand{\tocappendix}[3]{%
    \indentlabel{\@ifnotempty{#2}{\color{toccolor}\ignorespaces#1 #2.\quad}}{\bfseries\color{toccolor}#3}}

\def\l@section{\@tocline{1}{2pt}{0pt}{}{}}
\def\l@subsection{\@tocline{2}{0pt}{1.5em}{}{}}

\def\@tocline#1#2#3#4#5#6#7{\relax
    \ifnum #1>\c@tocdepth
    \else
        \par \addpenalty\@secpenalty\addvspace{#2}%
        \begingroup \hyphenpenalty\@M \small
        \@ifempty{#4}{%
            \@tempdima\csname r@tocindent\number#1\endcsname\relax
        }{%
            \@tempdima#4\relax
        }%
        \parindent\z@ \leftskip#3\relax \advance\leftskip\@tempdima\relax
        \rightskip\@pnumwidth plus4em \parfillskip-\@pnumwidth
        #5\leavevmode\hskip-\@tempdima
        {\color{toccolor}#6}\nobreak\relax
        \leaders\hbox{$\color{toccolor}\m@th
            \mkern 4.5mu\hbox{\normalfont\small\color{toccolor}.}%
            \mkern 4.5mu$}\hfill
        \hbox to\@pnumwidth{\@tocpagenum{\color{toccolor}#7}}\par
        \nobreak
        \endgroup
    \fi}
\makeatother

\usepackage{enumitem}
\setlist{leftmargin=*, itemsep=3pt, parsep=1pt}
\setlist[enumerate,1]{label=(\roman*)}
\setlist[enumerate,2]{label=(\alph*)}

\definecolor{abstractbg}{gray}{0.93}

\makeatletter
\renewenvironment{abstract}{%
    \ifx\maketitle\relax
        \ClassWarning{\@classname}{Abstract should precede
            \protect\maketitle}%
    \fi
    \global\setbox\abstractbox=\vtop\bgroup
        \normalfont\small
        \list{}{\labelwidth\z@
            \leftmargin5pc \rightmargin\leftmargin
            \listparindent\normalparindent \itemindent\z@
            \parsep\z@ \@plus\p@
            
        }%
        \item\relax
        \begin{center}\textbf{Abstract}\end{center}%
        \vspace{-2pt}%
        \noindent\ignorespaces
}{%
    \endlist\egroup
    \ifx\@setabstract\relax \@setabstracta \fi
}

\let\orig@setabstracta\@setabstracta
\renewcommand{\@setabstracta}{%
    \orig@setabstracta
}
\makeatother

\allowdisplaybreaks[1]               

\title[Asymptotic counting with restricted prime factors]{Asymptotic counting of integers\\
with prime factors $p_{r^a s^b}$}

\author{Mehdi Golafshan}
\address{Department of Mathematics, University of Li\`ege,
         All\'ee de la D\'ecouverte 12 (B37), 4000 Li\`ege, Belgium}
\email{mgolafshan@uliege.be}
\thanks{Supported by the FNRS Research grant T.196.23 (PDR)}

\date{}

\begin{document}

\begin{abstract}
Let $p_n$ be the $n$th prime, and let $r,s\ge2$ be fixed
multiplicatively independent integers. We count the integers up to $x$
whose prime factors all have the form $p_{r^a s^b}$ with integers
$a,b\ge0$. Our asymptotic formula for this count has relative error
$o(1)$ and is explicit down to the multiplicative constant. For
$(r,s)=(2,3)$, these integers are the prime codes of ordinals below
$\omega^{\omega^2}$, and our formula gives an explicit asymptotic
equivalent for their counting function.
\end{abstract}

\maketitle

\vspace{-4pt}
\begin{list}{}{%
    \leftmargin3.5pc \rightmargin3.5pc
    \listparindent0pt \itemindent0pt
    \parsep4pt \topsep0pt \partopsep0pt}
\item\relax

\item\relax
\end{list}
\vspace{4pt}


\section{Introduction}
\label{sec:intro}

Let $p_n$ denote the $n$th prime and $\mathbb N=\{0,1,2,\ldots\}$.
Fix multiplicatively independent integers $r,s\ge2$ and put
$\mathcal P_{r,s}=\{p_{r^a s^b}:a,b\in\mathbb N\}$. We call $r$ and $s$
the generators of the index set. Let $S_{r,s}$ be the set of finite
products of primes in $\mathcal P_{r,s}$, including the empty product $1$. For $x\ge1$, let
\[
C_{r,s}(x)=\#\{n\in S_{r,s}:n\le x\}.
\]
We seek \emph{strong asymptotics} for $C_{r,s}(x)$, that is, an asymptotic
equivalent as $x\to\infty$. In what follows, $r$ and $s$ are fixed, and we
write $\mathcal P$, $S$ and $C(x)$.

Vernaeve et al.~\cite{VVW2014} solved the one-generator problem, with
prime indices $m^j$, where $m\ge2$ is a fixed integer and $j\in\mathbb N$.
In 2017, Neyt~\cite{Neyt2017} announced strong asymptotics for
multiplicatively generated prime-index sets, including the
two-generator case. His approach combines distributional expansions
and Ingham's Tauberian theorem~\cite[slides~10--15]{Neyt2017} with a
finite iterative approximation of the saddle
point~\cite[slides~16--22]{Neyt2017}. In these slides, the asymptotic
formula is expressed through this approximation of the saddle point.
The lower-order coefficients are left symbolic, and no explicit
expression is given for the multiplicative constant.

We are not aware of a published proof of the two-generator case.
We give a complete Mellin-transform proof for two multiplicatively
independent generators and determine the lower-order coefficients
explicitly, together with a convergent expression for the
multiplicative constant. \Cref{thm:prime-index-strong-asymptotic}
states the result directly in terms of $\log x$.

We expand the prime weights $\log p_{r^a s^b}$ around the lattice
weights $a\log r+b\log s$. This expresses their Dirichlet series in
terms of a Barnes-type lattice zeta function~\cite{Ruijsenaars2000}
and yields a sectorial Mellin expansion. Ingham's theorem and a precise
expansion of the saddle point then give the asymptotic equivalent.

\secref{subsec:geo-prime-weights} gives a geometric
interpretation of the leading terms in the Mellin expansion.
Counting pairs
$(a,b)\in\mathbb N^2$ with $a\log r+b\log s\le t$ amounts
to counting lattice points in a right triangle with irrational
slope, a problem studied by Hardy and Littlewood~\cite{HL22}.
Everest~\cite{Eve92} considered related counting problems
for regions defined by linear inequalities.
We obtain a two-term estimate for this count by an
equidistribution argument. For the prime weights, the same
count applies with $t$ replaced by $\log\pi(e^t)$,
where $\pi$ denotes the prime-counting function.

When $(r,s)=(2,3)$, the function $C(x)$ also counts the prime codes
not exceeding $x$ of ordinals below $\omega^{\omega^2}$.
Vernaeve et al.~\cite[Sec.~2.2]{VVW2014} ask for strong asymptotics
for the prime codes of ordinals below $\omega^{\omega^k}$, for each
fixed $k\ge1$. \Cref{thm:prime-index-strong-asymptotic} answers this
question for $k=2$. \secref{sec:trees-ordinals} describes the
ordinal coding and the corresponding rooted trees.

Implied constants may depend on $r$ and $s$. Any further
dependence is indicated by subscripts. All logarithms are
natural. We write $\zeta$ for the Riemann zeta function and
$\gamma$ for Euler's constant. The Stieltjes constant
$\gamma_1$ is normalised by
\[
\zeta(1+z)
=\frac1z+\gamma-\gamma_1z+O(z^2)
\qquad\text{as }z\to0.
\]

Define
\[
\begin{aligned}
A_1&=\frac{\zeta(3)}{\log r\log s},\qquad
A_2=\frac{\zeta(2)}{\log r\log s},\qquad
A_3=A_2\left(\frac{\log(rs)}2+\gamma-1\right)-\frac{\zeta'(2)}{\log r\log s},\\
\nu&=\frac{(\log r)^2+(\log s)^2+\pi^2+12-12\gamma_1-6\gamma^2}
{12\log r\log s}-\frac34+\frac{A_2(A_3-A_2)}{6A_1}.
\end{aligned}
\]

{\boldthmnote
\begin{theorem}[Main result]
\label{thm:prime-index-strong-asymptotic}
Let $r,s\ge2$ be fixed multiplicatively independent integers. For
${x>1}$, let ${W=(\log x/(2A_1))^{1/3}}$. Then
\begin{equation}
\label{eq:main-theorem-strong-asymptotic}
\begin{aligned}
C(x)\sim K\,W^\nu\exp\bigg(&3A_1W^2-A_2W\log W+A_3W
+\frac{(\log W)^3}{6\log r\log s}\\
&-\left(\frac{\log(rs)}{4\log r\log s}
+\frac{A_2^2}{12A_1}\right)(\log W)^2\bigg)
\qquad\text{as }x\to\infty,
\end{aligned}
\end{equation}
where $K>0$ is given explicitly in \secref{sec:main-proof}.
\end{theorem}}

In particular,
\[
\log C(x)\sim
3\left(\frac{\zeta(3)}{4\log r\log s}\right)^{1/3}(\log x)^{2/3}.
\]

An asymptotic equivalent for $C(x)$ requires an additive error $o(1)$
in $\log C(x)$. Hence every term, including the constant $\log K$,
must be kept.

Multiplicative independence ensures that the indices $r^a s^b$ are
distinct, but the generators need not be coprime.

The result of Vernaeve et al.~\cite[Thm.~2]{VVW2014} does not apply
directly to the present weights.
Applied to the counting function $N_\lambda$ from
\secref{subsec:geo-prime-weights}, the theorem would require
\[
\int_0^u N_\lambda(t)\,\frac{\mathrm dt}{t}
=\frac{A}{\alpha}u^\alpha
+\beta_2\log^2u+\beta_1\log u+\beta_0+o(1)
\qquad\text{as }u\to\infty,
\]
for some constants $\alpha,A>0$ and $\beta_0,\beta_1,\beta_2\in\mathbb R$.
In our case, \cref{cor:geo-prime-count} gives
\[
\int_0^u N_\lambda(t)\,\frac{\mathrm dt}{t}
=\frac{u^2-4u\log u+2(\log(rs)+2)u}{4\log r\log s}+o(u)
\qquad\text{as }u\to\infty,
\]
so the terms of order $u\log u$ and $u$ are incompatible with
this hypothesis.

Taking logarithms of the prime factorisations identifies
$C(e^t)$ with the number of partitions of total size at most $t$
into parts $\lambda_{a,b}=\log p_{r^a s^b}$.
The generating series is
\[
\mathcal F(\tau)=\sum_{n\in S}n^{-\tau}.
\]
By \cref{prop:mellin-expansion}, the expansion of $\log\mathcal F(\tau)$
as $\tau\to0^+$ contains the term $-A_2\tau^{-1}\log(1/\tau)$ and a
cubic polynomial in $\log(1/\tau)$.

For partitions into integer parts, Granovsky and
Stark~\cite{GranovskyStark2012} obtained asymptotics when the associated
Dirichlet series has several simple positive poles. Debruyne and
Tenenbaum~\cite{DebruyneTenenbaum2020} developed a saddle-point method
for suitable sets of integer parts, and Bridges
et al.~\cite{BridgesBrindleBringmannFranke2024} gave asymptotic
expansions for partitions generated by a class of infinite products.
In our problem the parts $\lambda_{a,b}$ are not integers. Moreover,
the Mellin kernel of $\log\mathcal F$ has a double pole at $1$ and a
pole of order four at $0$. Condition~(P2)
in~\cite[Thm.~1.4]{BridgesBrindleBringmannFranke2024} allows only
simple poles away from zero and at most a double pole at zero. Hence
this theorem does not apply directly.

The paper is organised as follows:
\begin{itemize}
\item \secref{sec:interpretations} describes the tree, ordinal and
lattice-point interpretations;
\item \secref{sec:prime-weights,sec:mellin} establish the
prime-weight and Mellin expansions;
\item \secref{sec:saddle} applies Ingham's theorem and expands the
saddle point;
\item \secref{sec:main-proof} proves
\cref{thm:prime-index-strong-asymptotic} and gives $K$ explicitly;
\item \secref{sec:concluding} returns to the case $(r,s)=(2,3)$.
\end{itemize}

\section{Combinatorial and geometric interpretations}
\label{sec:interpretations}

This section explains the combinatorial origin of the problem and gives
a geometric explanation of the first logarithmic correction. Its results
are not used in the proof of \cref{thm:prime-index-strong-asymptotic}.

\subsection{Trees and ordinals}
\label{sec:trees-ordinals}

Matula~\cite{Matula1968} introduced a bijection between rooted trees
and the positive integers in 1968. G\"obel~\cite{Gobel1980}
rediscovered it in 1980. Under this bijection, our counting problem
becomes a counting problem for rooted trees. We follow the presentation
in~\cite[Sec.~2.1]{VVW2014}. All trees are finite and rooted, and
trees related by a root-preserving isomorphism are identified.

We denote the one-vertex tree by $\bullet$. For any other tree $T$, let
$T_1,\ldots,T_q$ be the subtrees rooted at the children of its root.
The Matula number of $T$ is defined recursively by
\[
M(\bullet)=1
\quad\text{and}\quad
M(T)=\prod_{i=1}^{q}p_{M(T_i)}.
\]
Unique prime factorisation makes $M$ a bijection between
these trees and the positive integers.

For $(r,s)=(2,3)$, the trees counted by $C(x)$ have a simple
description. For $a,b\in\mathbb N$, let $T_{a,b}$ be the tree whose
root has $a$ leaf children and $b$ further children, each with a
single leaf child. The subtrees at these children have Matula numbers
$1$ and $2$, respectively, and hence
\[
M(T_{a,b})=p_1^a p_2^b=2^a3^b.
\]
Attaching finitely many trees $T_{a_i,b_i}$ to a new root gives a tree
with Matula number $\prod_i p_{2^{a_i}3^{b_i}}$. These are precisely
the trees of height at most three in which every vertex at depth two
has at most one child. Here depth and height are measured in edges,
and the empty product gives the one-vertex tree. Hence $C(x)$ is the
number of such trees $T$ with $M(T)\le x$.
\figref{fig:matula} illustrates both constructions.

\begin{figure}[H]
\centering
\begin{tikzpicture}[line cap=round,line join=round,
  v/.style={circle,fill=geonavy,draw=white,line width=0.5pt,
            inner sep=0pt,minimum size=5pt},
  e/.style={line width=0.65pt,draw=geonavy},
  dot/.style={circle,fill=geonavy!75,inner sep=0pt,minimum size=1.3pt},
  br/.style={decorate,decoration={brace,amplitude=3.5pt,mirror},
             line width=0.45pt},
  gd/.style={line width=0.3pt,draw=black!14},
  box/.style={fill=geogold!20,rounded corners=3pt},
  lb/.style={font=\footnotesize,inner sep=1.5pt},
  num/.style={font=\scriptsize,inner sep=1.5pt},
  dp/.style={font=\scriptsize,text=black!60,inner sep=1pt}]
  \foreach \y/\t in {0/0,-1/1,-2/2,-3/3}
    {\draw[gd] (-2.55,\y) -- (10.25,\y);
     \node[dp,right=3pt] at (10.25,\y) {depth $\t$};}
  \node[lb,anchor=west] at (-2.55,0.40) {(a)};
  \node[v] (r) at (0,0) {};
  \node[lb,above=3pt] at (r) {$p_1^{a}\,p_2^{b}=2^a3^b$};
  \foreach \x in {-2.00,-1.50,-0.50}
    {\node[v] (l) at (\x,-1) {}; \draw[e] (r) -- (l);
     \node[num,below=1.5pt] at (l) {$1$};}
  \foreach \x in {-1.10,-1.00,-0.90} {\node[dot] at (\x,-1) {};}
  \foreach \x in {0.50,1.00,2.00}
    {\node[v] (m) at (\x,-1) {}; \node[v] (n) at (\x,-2) {};
     \draw[e] (r) -- (m); \draw[e] (m) -- (n);
     \node[num,right=2pt] at (m) {$2$};}
  \foreach \x in {1.40,1.50,1.60} {\node[dot] at (\x,-1.5) {};}
  \draw[br] (-2.20,-1.45) -- (-0.30,-1.45)
     node[midway,below=4pt,lb] {$a$};
  \draw[br] (0.30,-2.25) -- (2.20,-2.25)
     node[midway,below=4pt,lb] {$b$};
  \begin{scope}[xshift=6.85cm]
    \node[lb,anchor=west] at (-2.875,0.40) {(b)};
    \fill[box] (-2.875,-0.68) rectangle (-1.275,-3.27);
    \fill[box] (-0.75,-0.68) rectangle (0.80,-3.27);
    \fill[box] (1.325,-0.68) rectangle (3.275,-3.27);
    \foreach \y in {-1,-2,-3} {\draw[gd] (-2.875,\y) -- (3.275,\y);}
    \node[v] (r) at (0,0) {};
    \node[lb,above=3pt] at (r) {$p_{12}\,p_{18}\,p_{36}$};
    \node[v] (s1) at (-2.075,-1) {}; \draw[e] (r) -- (s1);
    \node[num,left=2pt] at (s1) {$12$};
    \foreach \x in {-2.525,-2.075}
      {\node[v] (t) at (\x,-2) {}; \draw[e] (s1) -- (t);}
    \node[v] (p) at (-1.625,-2) {}; \node[v] (q) at (-1.625,-3) {};
    \draw[e] (s1) -- (p); \draw[e] (p) -- (q);
    \node[lb] at (-2.075,-3.55) {$T_{2,1}$};
    \node[v] (s2) at (0,-1) {}; \draw[e] (r) -- (s2);
    \node[num,right=2pt] at (s2) {$18$};
    \node[v] (t) at (-0.45,-2) {}; \draw[e] (s2) -- (t);
    \foreach \x in {0,0.45}
      {\node[v] (p) at (\x,-2) {}; \node[v] (q) at (\x,-3) {};
       \draw[e] (s2) -- (p); \draw[e] (p) -- (q);}
    \node[lb] at (0,-3.55) {$T_{1,2}$};
    \node[v] (s3) at (2.30,-1) {}; \draw[e] (r) -- (s3);
    \node[num,right=2pt] at (s3) {$36$};
    \foreach \x in {1.625,2.075}
      {\node[v] (t) at (\x,-2) {}; \draw[e] (s3) -- (t);}
    \foreach \x in {2.525,2.975}
      {\node[v] (p) at (\x,-2) {}; \node[v] (q) at (\x,-3) {};
       \draw[e] (s3) -- (p); \draw[e] (p) -- (q);}
    \node[lb] at (2.30,-3.55) {$T_{2,2}$};
  \end{scope}
\end{tikzpicture}
\caption{(a) The tree $T_{a,b}$. (b) The tree obtained by attaching
$T_{2,1}$, $T_{1,2}$ and $T_{2,2}$ to a new root. The number at a vertex
is the Matula number of the subtree rooted there.}
\label{fig:matula}
\end{figure}

The same recursion gives the ordinal coding used
in~\cite[Sec.~2.2]{VVW2014}. Here $\omega$ denotes the first infinite
ordinal, and the code of $0$ is $1$. A Cantor normal form is coded by
multiplying the primes indexed by the codes of its exponents.
In terms of trees, a leaf represents $0$, and the subtrees at the root
represent the exponents.
The tree $T_{a,b}$ therefore represents $\omega b+a$.
Since every ordinal below $\omega^{\omega^2}$ is a finite sum of terms
of the form $\omega^{\omega b+a}$ in Cantor normal form, these ordinals
correspond bijectively to finite products of the primes $p_{2^a3^b}$,
with the empty product representing $0$.
Hence, for $(r,s)=(2,3)$, the function $C(x)$ is also the number of ordinals below
$\omega^{\omega^2}$ whose code does not exceed $x$.

\subsection{Lattice-point geometry of the prime weights}
\label{subsec:geo-prime-weights}

A lattice-point estimate gives a geometric interpretation
of the leading terms in the Mellin expansion.

For $a,b\in\mathbb N$, put $u_{a,b}=a\log r+b\log s$ and
$\lambda_{a,b}=\log p_{r^a s^b}$. For $t\ge0$, let
\[
N_u(t)=\#\{(a,b)\in\mathbb N^2:u_{a,b}\le t\}
\quad\text{and}\quad
N_\lambda(t)=\#\{(a,b)\in\mathbb N^2:\lambda_{a,b}\le t\}.
\]
In other words, $N_u(t)$ counts the integer points in the closed
triangle
\[
\{(x,y)\in\mathbb R^2:x,y\ge0,\ x\log r+y\log s\le t\}.
\]
Its area, $t^2/(2\log r\log s)$, is the leading term in
\cref{prop:geo-index-count}. Because the indices $r^a s^b$ are distinct,
$N_\lambda(t)$ is the number of primes in $\mathcal P$ not exceeding $e^t$.

\begin{proposition}
\label{prop:geo-index-count}
The number of integer points in the triangle satisfies
\[
N_u(t)=\frac{t^2+t\log(rs)}{2\log r\log s}+o(t)
\qquad\text{as }t\to\infty.
\]
\end{proposition}

\begin{proof}
Put $q=\lfloor t/\log r\rfloor$ and $\theta=t/\log r-q$.
Counting the points on the vertical lines $x=a$ gives
\[
N_u(t)=\sum_{a=0}^{q}
\left(\left\lfloor\frac{t-a\log r}{\log s}\right\rfloor+1\right).
\]
Let $\psi(x)=x-\lfloor x\rfloor-1/2$ and $\mu=\log r/\log s$.
Since $\lfloor x\rfloor+1=x+1/2-\psi(x)$, summing over $a$ gives
\[
N_u(t)=\frac{t^2+t\log(rs)}{2\log r\log s}
+\frac{1-\theta}{2}+\frac{\mu\theta(1-\theta)}{2}
-\sum_{a=0}^{q}\psi\Bigl(\frac{t}{\log s}-a\mu\Bigr).
\]
Since $0\le\theta<1$, the second and third terms are bounded.
The number $\mu$ is irrational, since $r^n\ne s^m$ for all positive
integers $m$ and $n$.
Hence, for every nonzero integer $h$ and every $q$,
\[
\sup_{y\in\mathbb R}\biggl|\sum_{a=0}^{q}e^{2\pi ih(y-a\mu)}\biggr|
\le\frac{2}{|1-e^{2\pi ih\mu}|}.
\]
The function $\psi$ is $1$-periodic and Riemann integrable on $[0,1]$,
with integral zero. Fix $0<\varepsilon<1$. As in the proof of Weyl's
criterion~\cite{Wey16}, there exist real trigonometric polynomials
$P_-\le\psi\le P_+$ with
$\bigl|\int_0^1P_\pm(x)\,\mathrm{d}x\bigr|<\varepsilon$.
To construct them, let $\delta=\varepsilon/3$. Let $g_+$ be the
continuous $1$-periodic function that agrees with $\psi$ for
$\delta\le x<1$ and is linear for $0\le x\le\delta$, with $g_+(0)=1/2$.
Let $g_-$ be the continuous $1$-periodic function that agrees with
$\psi$ for $0\le x\le1-\delta$ and is linear for $1-\delta\le x\le1$,
with $g_-(1)=-1/2$. Then $g_-\le\psi\le g_+$ on $\mathbb R$, and
$\int_0^1g_\pm(x)\,\mathrm{d}x=\pm\delta/2$. By the Weierstrass
approximation theorem, there exist real trigonometric polynomials
$T_\pm$ with $|T_\pm-g_\pm|\le\delta$ on $\mathbb R$. The polynomials
$P_\pm=T_\pm\pm\delta$ then satisfy $P_-\le g_-\le\psi\le g_+\le P_+$ and
${\bigl|\int_0^1P_\pm(x)\,\mathrm{d}x\bigr|\le5\delta/2<\varepsilon}$.

Applying the preceding bound to the nonconstant Fourier modes of
$P_+$ and $P_-$ gives
\[
\sup_{y\in\mathbb R}\biggl|\sum_{a=0}^{q}\psi(y-a\mu)\biggr|
\le\varepsilon(q+1)+O_\varepsilon(1).
\]
Since $\varepsilon$ is arbitrary, the left-hand side is $o(q)$ as
$q\to\infty$.
With $y=t/\log s$, it follows that the sum involving $\psi$ in the
expansion of $N_u(t)$ is $o(q)$. As $q\asymp t$, this sum is $o(t)$,
which proves the proposition.
\end{proof}

Let $\pi(y)$ denote the number of primes not exceeding $y$. Assume
that $t\ge\log2$, so that $\log\pi(e^t)\ge0$. Since $p_n\le y$ if and only if
$n\le\pi(y)$, we have
\[
\lambda_{a,b}\le t\iff p_{r^a s^b}\le e^t\iff r^a s^b\le\pi(e^t)
\iff u_{a,b}\le\log\pi(e^t).
\]
Hence $N_\lambda(t)=N_u(\log\pi(e^t))$. Geometrically, the pairs $(a,b)$
with $\lambda_{a,b}\le t$ are the integer points of the smaller triangle
obtained by replacing $t$ with $\log\pi(e^t)$.

\begin{corollary}
\label{cor:geo-prime-count}
The number of prime weights up to $t$ satisfies
\[
N_\lambda(t)=\frac{t^2+t\log(rs)}{2\log r\log s}
-\frac{t\log t}{\log r\log s}+o(t)
\qquad\text{as }t\to\infty.
\]
\end{corollary}

\begin{proof}
By the prime number theorem, $\log\pi(e^t)=t-\log t+o(1)$ as
$t\to\infty$. Applying \cref{prop:geo-index-count} with $t$ replaced by
$\log\pi(e^t)$ gives the corollary, since $(\log t)^2=o(t)$.
\end{proof}

By \cref{prop:geo-index-count,cor:geo-prime-count},
\[
N_u(t)-N_\lambda(t)=\frac{t\log t}{\log r\log s}+o(t)
\qquad\text{as }t\to\infty.
\]
Up to $o(t)$, this is also the area between the two triangles, since
$\log\pi(e^t)=t-\log t+o(1)$.
\figref{fig:geo-prime-weights} shows the case $(r,s)=(2,3)$ and $t=8$.

\input{lattice_prime_weights_figure}

\section{Prime weights and Dirichlet series}
\label{sec:prime-weights}

For $v>0$ and $z\in\mathbb C$, we write $v^{-z}=\exp(-z\log v)$.
We also write $z=\sigma+it$ with $\sigma,t\in\mathbb R$.

\subsection{The prime weights}
\label{subsec:prime-weights}

Recall that $u_{a,b}=a\log r+b\log s$ and
$\lambda_{a,b}=\log p_{r^a s^b}$ for $a,b\in\mathbb N$.
Then
\[
\min\{\log r,\log s\}(a+b)\le u_{a,b}
\le\max\{\log r,\log s\}(a+b).
\]

We recall the classical asymptotic expansion of the $n$th prime, as
in~\cite[Eq.~(6.2)]{AriasDeReynaToulisse2013}:
\[
p_n=n\left(\log n+\log\log n-1
+O\!\left(\frac{\log\log n}{\log n}\right)\right)
\qquad\text{as }n\to\infty.
\]
Taking logarithms with $n=r^a s^b$ gives
\begin{equation}
\label{eq:prime-weight-real-expansion}
\lambda_{a,b}=u_{a,b}+\log u_{a,b}+\frac{\log u_{a,b}-1}{u_{a,b}}
+O((\log u_{a,b})^2u_{a,b}^{-2})
\qquad\text{as }u_{a,b}\to\infty.
\end{equation}
For complex $z$, the next lemma gives the corresponding expansion of
$\lambda_{a,b}^{-z}$ as $u_{a,b}\to\infty$. The bound on the remainder
is uniform in $(a,b)$ and in $z$ on vertical strips.
We retain terms through order $u_{a,b}^{-z-2}$, so that the summed
remainder is holomorphic in $\sigma>-1$. In \secref{sec:mellin},
this allows us to move the line of integration to the left of the
imaginary axis.

\begin{lemma}
\label{lem:prime-weight-expansion}
Let $a,b\in\mathbb N$ with $(a,b)\ne(0,0)$ and let $z\in\mathbb C$. Then
\begin{equation}
\label{eq:prime-weight-complex-expansion}
\lambda_{a,b}^{-z}
=u_{a,b}^{-z}
-z(\log u_{a,b})u_{a,b}^{-z-1}
+\left[
\frac{z(z+1)}2(\log u_{a,b})^2
-z\log u_{a,b}+z
\right]u_{a,b}^{-z-2}
+R_{a,b}(z),
\end{equation}
where $R_{a,b}$ is entire. For all real $\sigma_-\le\sigma_+$ and
uniformly in $(a,b)$,
\[
|R_{a,b}(z)|
\ll_{\sigma_-,\sigma_+}
(1+|z|)^3(1+|\log u_{a,b}|)^3u_{a,b}^{-\sigma-3}
\qquad\text{for }\sigma_-\le\sigma\le\sigma_+.
\]
\end{lemma}

\begin{proof}
Fix a strip $\sigma_-\le\Re z\le\sigma_+$. Since $u_{a,b}>0$ and
$\lambda_{a,b}>0$, every power in \eqref{eq:prime-weight-complex-expansion}
is an entire function of $z$, and hence so is $R_{a,b}$.

Set $\eta=\lambda_{a,b}/u_{a,b}-1$. By \eqref{eq:prime-weight-real-expansion},
\[
\begin{aligned}
\eta
&=\frac{\log u_{a,b}}{u_{a,b}}
+\frac{\log u_{a,b}-1}{u_{a,b}^2}
+O\!\left(\frac{(\log u_{a,b})^2}{u_{a,b}^3}\right),\\
\eta^2
&=\frac{(\log u_{a,b})^2}{u_{a,b}^2}
+O\!\left(\frac{(\log u_{a,b})^2}{u_{a,b}^3}\right).
\end{aligned}
\]
For large $u_{a,b}$, we have $|\eta|\le1/2$, and hence
$1/2\le1+\theta\eta\le3/2$ for $0\le\theta\le1$. Taylor's formula with
integral remainder, applied to $x\mapsto(1+x)^{-z}$ at $0$, gives
\[
(1+\eta)^{-z}
=1-z\eta+\frac{z(z+1)}2\eta^2
-\frac{z(z+1)(z+2)}2\eta^3
\int_0^1(1-\theta)^2
(1+\theta\eta)^{-z-3}\,\mathrm d\theta.
\]
Since $1/2\le1+\theta\eta\le3/2$,
\[
\bigl|(1+\theta\eta)^{-z-3}\bigr|
=(1+\theta\eta)^{-\sigma-3}
\ll_{\sigma_-,\sigma_+}1
\qquad\text{for }0\le\theta\le1.
\]
Since $\eta\ll(\log u_{a,b})/u_{a,b}$, the remainder term in Taylor's
formula is
\[
O_{\sigma_-,\sigma_+}\!\left(
(1+|z|)^3\frac{(\log u_{a,b})^3}{u_{a,b}^3}
\right).
\]
Since $\lambda_{a,b}^{-z}=u_{a,b}^{-z}(1+\eta)^{-z}$, inserting the
expansions of $\eta$ and $\eta^2$ into Taylor's formula proves the
bound on $R_{a,b}$ for large $u_{a,b}$.
For each of the finitely many remaining pairs, we have
$R_{a,b}(z)=O_{\sigma_-,\sigma_+}((1+|z|)^2)$
by \eqref{eq:prime-weight-complex-expansion}. Since
$u_{a,b}^{-\sigma-3}\gg_{\sigma_-,\sigma_+}1$ for these pairs, they are
covered by enlarging the implied constant.
\end{proof}

For $\Re z>-1$, set
\[
\mathcal H(z)=\sum_{(a,b)\ne(0,0)}R_{a,b}(z).
\]
Fix real numbers $-1<\sigma_-\le\sigma_+$. For $m\ge1$, there are $m+1$
pairs $(a,b)$ with $a+b=m$, and $u_{a,b}\asymp m$ for each of them.
Grouping the pairs by $m$ and applying \cref{lem:prime-weight-expansion},
we obtain
\[
\sum_{(a,b)\ne(0,0)}|R_{a,b}(z)|
\ll_{\sigma_-,\sigma_+}
(1+|z|)^3
\sum_{m\ge1}
\frac{(m+1)(1+\log m)^3}{m^{\sigma_-+3}}
\qquad\text{for }\sigma_-\le\sigma\le\sigma_+.
\]
The numerical series converges because $\sigma_->-1$.
Hence the series defining $\mathcal H$ converges absolutely
and locally uniformly in $\Re z>-1$.
Its sum is holomorphic there, and
\[
|\mathcal H(z)|
\ll_{\sigma_-,\sigma_+}(1+|z|)^3
\qquad\text{for }\sigma_-\le\sigma\le\sigma_+.
\]

\subsection{The lattice zeta function}

Let $\Gamma$ denote the gamma function. We define the lattice zeta
function by
\[
Z(z)=\sum_{(a,b)\ne(0,0)}u_{a,b}^{-z}
\qquad\text{for }\Re z>2.
\]
We also put
\[
c_2=\frac1{\log r\log s},
\qquad
c_1=\frac{\log(rs)}{2\log r\log s}
\qquad\text{and}\qquad
c_0=\frac{\log r}{12\log s}
+\frac{\log s}{12\log r}-\frac34.
\]

\begin{lemma}
\label{lem:lattice-zeta-continuation}
The function $Z$ extends meromorphically to $\Re z>-1$. Its only poles
there are simple poles at $1$ and $2$, with residues $c_1$ and $c_2$,
respectively. Moreover, $Z(0)=c_0$.
\end{lemma}

\begin{proof}
Since $u_{a,b}\asymp a+b$, the series defining $Z$ converges absolutely
and locally uniformly in $\Re z>2$. For $\Re z\ge\sigma_0>2$, it is
dominated by a constant multiple of $\sum_{m\ge1}(m+1)m^{-\sigma_0}$.
For $y>0$, put
\[
\Theta(y)=\sum_{(a,b)\ne(0,0)}e^{-u_{a,b}y}
=\frac1{(1-r^{-y})(1-s^{-y})}-1.
\]
We use the expansion
\[
\frac1{1-e^{-\xi}}=\frac1\xi+\frac12+\frac\xi{12}+O(\xi^3)
\qquad\text{as }\xi\to0.
\]
Applying it with $\xi=y\log r$ and $\xi=y\log s$ gives
\[
\Theta(y)=c_2y^{-2}+c_1y^{-1}+c_0+O(y)
\qquad\text{as }y\to0^+.
\]
For large $y$, on the other hand, $\Theta(y)\ll\min\{r,s\}^{-y}$.
Interchanging the absolutely convergent sum and integral gives
\[
\Gamma(z)Z(z)=\int_0^\infty y^{z-1}\Theta(y)\,\mathrm dy
\qquad\text{for }\Re z>2.
\]
We split the integral at $y=1$. Over $(0,1)$, we subtract
$c_2y^{-2}+c_1y^{-1}+c_0$ from $\Theta(y)$ and integrate these terms
explicitly. For $\Re z>2$, this gives
\begin{equation}
\label{eq:lattice-zeta-representation}
\begin{aligned}
\Gamma(z)Z(z)
={}&\frac{c_2}{z-2}+\frac{c_1}{z-1}+\frac{c_0}{z}
+\int_1^\infty y^{z-1}\Theta(y)\,\mathrm dy\\
&+\int_0^1 y^{z-1}
\bigl(\Theta(y)-c_2y^{-2}-c_1y^{-1}-c_0\bigr)\,\mathrm dy.
\end{aligned}
\end{equation}
By the exponential decay of $\Theta$, the integral over $[1,\infty)$
defines an entire function of $z$. The integral over $(0,1)$ is
holomorphic in $\Re z>-1$, since its integrand is $O(y^{\sigma})$ as
$y\to0^+$.
Since $1/\Gamma$ is entire, dividing
\eqref{eq:lattice-zeta-representation} by $\Gamma(z)$ continues $Z$
meromorphically to $\Re z>-1$, with at most simple poles at $z=0,1,2$. Since $\Gamma(1)=\Gamma(2)=1$,
the residues at $1$ and $2$ are $c_1$ and $c_2$.
Near $z=0$, the right-hand side of
\eqref{eq:lattice-zeta-representation} is $c_0/z+O(1)$ and
$1/\Gamma(z)=z+O(z^2)$. Hence $Z$ is holomorphic at $0$, with
$Z(0)=c_0$.
\end{proof}

The next lemma shows that $Z$ and its derivatives grow at most
polynomially on vertical strips. \secref{sec:mellin} needs this to
move the line of integration.

\begin{lemma}
\label{lem:lattice-zeta-growth}
Let $j\in\mathbb N$ and $-1<\sigma_-<\sigma_+$. For $z=\sigma+it$ with
$\sigma_-\le\sigma\le\sigma_+$ and $|t|\ge1$, we have
\[
|Z^{(j)}(z)|\ll_{j,\sigma_-,\sigma_+}(1+|t|)^4.
\]
\end{lemma}

\begin{proof}
For $q,h>0$ and $\Re z>1$, put $f(v)=(q+hv)^{-z}$.
Let $\{v\}$ be the fractional part of $v$. We use the periodic
Bernoulli functions
\[
\widetilde B_2(v)=\{v\}^2-\{v\}+\frac16
\qquad\text{and}\qquad
\widetilde B_3(v)=\{v\}^3-\frac32\{v\}^2+\frac12\{v\}.
\]
We apply Euler--Maclaurin summation to $\sum_{n=0}^{M}f(n)$ and let the
integer $M$ tend to infinity. Since $\Re z>1$, the boundary terms at $M$
vanish in the limit, and we obtain
\[
\sum_{n\ge0}f(n)
=\int_0^\infty f(v)\,\mathrm dv
+\frac{f(0)}2-\frac{f'(0)}{12}
-\frac12\int_0^\infty
\widetilde B_2(v)f''(v)\,\mathrm dv.
\]
Since $\widetilde B_3$ is differentiable on $\mathbb R$ with
$\widetilde B_3'=3\widetilde B_2$, integration by parts gives
\[
-\frac12\int_0^\infty
\widetilde B_2(v)f''(v)\,\mathrm dv
=
\frac16\int_0^\infty
\widetilde B_3(v)f'''(v)\,\mathrm dv.
\]
The boundary term at $0$ vanishes because $\widetilde B_3(0)=0$. The
one at infinity vanishes because $\widetilde B_3$ is bounded and
$f''(v)\to0$ as $v\to\infty$.
Inserting the derivatives of $f$ now gives
\[
\begin{aligned}
\sum_{n\ge0}(q+hn)^{-z}
={}&\frac{q^{1-z}}{h(z-1)}
+\frac{q^{-z}}2+\frac{hz}{12}q^{-z-1}\\
&-\frac{h^3z(z+1)(z+2)}6
\int_0^\infty
\widetilde B_3(v)(q+hv)^{-z-3}\,\mathrm dv.
\end{aligned}
\]

We group the terms of $Z$ according to $a$. For $a=0$, the sum over
$b\ge1$ equals $(\log s)^{-z}\zeta(z)$. For $a\ge1$, the sum over $b\ge0$
is $\sum_{n\ge0}(q+hn)^{-z}$ with $q=a\log r$ and $h=\log s$, which was
computed above.
Summing over $a$, we obtain, for $\Re z>2$,
\[
\begin{aligned}
Z(z)
={}&\left((\log s)^{-z}
+\frac{(\log r)^{-z}}2\right)\zeta(z)
+\frac{(\log r)^{1-z}}{(\log s)(z-1)}\zeta(z-1)\\
&+\frac{z\log s}{12}(\log r)^{-z-1}\zeta(z+1)
+\mathcal E(z),
\end{aligned}
\]
where
\[
\mathcal E(z)
=-\frac{(\log s)^3z(z+1)(z+2)}6
\sum_{a\ge1}\int_0^\infty
\widetilde B_3(v)
(a\log r+v\log s)^{-z-3}\,\mathrm dv.
\]
To bound $\mathcal E$, we compute, for $\sigma>-1$,
\[
\sum_{a\ge1}\int_0^\infty
(a\log r+v\log s)^{-\sigma-3}\,\mathrm dv
=
\frac{(\log r)^{-\sigma-2}}
{(\log s)(\sigma+2)}
\zeta(\sigma+2).
\]
Since $\widetilde B_3$ is bounded, this identity shows that the series
of integrals defining $\mathcal E$ converges absolutely and locally
uniformly in $\Re z>-1$.
Thus $\mathcal E$ is holomorphic in $\Re z>-1$, and
\[
\mathcal E(z)
=O_{\sigma_-,\sigma_+}\bigl((1+|z|)^3\bigr)
\qquad\text{for }\sigma_-\le\sigma\le\sigma_+.
\]
Both sides of the representation of $Z$ in terms of $\zeta$ and
$\mathcal E$ are meromorphic in $\Re z>-1$. By analytic continuation, this
representation holds throughout the half-plane.

Taking $q=h=1$ and writing $w$ for $z$, we obtain, for $\Re w>1$,
\[
\zeta(w)=\frac1{w-1}+\frac12+\frac w{12}
-\frac{w(w+1)(w+2)}6\int_0^\infty\widetilde B_3(v)(1+v)^{-w-3}\,\mathrm dv.
\]
The integral converges for $\Re w>-2$, where it is bounded by
$\|\widetilde B_3\|_\infty/(\Re w+2)$. This formula therefore continues
$\zeta$ meromorphically to $\Re w>-2$. Hence
$\zeta(w)=O((1+|\Im w|)^3)$ on each closed strip in $\Re w>-2$ with
$|\Im w|\ge1/2$. Applying this bound to the three values of $\zeta$ in
the representation of $Z$, we obtain, for $\sigma_-\le\sigma\le\sigma_+$
and $|t|\ge1/2$,
\[
Z(\sigma+it)
=O_{\sigma_-,\sigma_+}\bigl((1+|t|)^4\bigr).
\]

Finally, fix $0<\varrho<\min\{(\sigma_-+1)/2,1/2\}$. If $z=\sigma+it$
lies in the stated strip and $|t|\ge1$, then every $w$ with
$|w-z|\le\varrho$ satisfies $\sigma_--\varrho\le\Re w\le\sigma_++\varrho$
and $|\Im w|\ge1/2$. Since the poles of $Z$ are real, $Z$ is holomorphic
on this disk. As $\sigma_--\varrho>-1$, the preceding bound, applied with
$\sigma_--\varrho$ and $\sigma_++\varrho$ in place of $\sigma_-$ and
$\sigma_+$, gives $|Z(w)|\ll_{\sigma_-,\sigma_+}(1+|t|)^4$ there.
Cauchy's estimate now yields
\[
|Z^{(j)}(z)|
\le
\frac{j!}{\varrho^j}
\max_{|w-z|=\varrho}|Z(w)|
\ll_{j,\sigma_-,\sigma_+}(1+|t|)^4.
\qedhere
\]
\end{proof}

\subsection{The Dirichlet series of the prime weights}

We define the Dirichlet series of the prime weights by
\[
D(z)=\sum_{a,b\ge0}\lambda_{a,b}^{-z}
=(\log2)^{-z}
+\sum_{(a,b)\ne(0,0)}\lambda_{a,b}^{-z}
\qquad\text{for }\Re z>2.
\]
The term $(\log2)^{-z}$ comes from $\lambda_{0,0}=\log p_1=\log2$. It is
written separately because $u_{0,0}=0$, so
\cref{lem:prime-weight-expansion} does not apply to it.
We shall also use the series
\[
\mathscr R
=\sum_{(a,b)\ne(0,0)}
\left[
\log\frac{\lambda_{a,b}}{u_{a,b}}
-\frac{\log u_{a,b}}{u_{a,b}}
+\frac{
\frac12(\log u_{a,b})^2-\log u_{a,b}+1
}{u_{a,b}^2}
\right].
\]
The regular parts of $Z$ at its poles $1$ and $2$, shifted to the
origin, are
\[
Z_1(z)=Z(1+z)-\frac{c_1}{z}
\qquad\text{and}\qquad
Z_2(z)=Z(2+z)-\frac{c_2}{z}.
\]
By \cref{lem:lattice-zeta-continuation}, both extend holomorphically
to $z=0$. Finally, put
\[
\begin{aligned}
d_0&=1+c_0+c_2,\\
d_1&=-\log(\log2)+Z'(0)+Z_1'(0)+Z_2(0)+Z_2'(0)
+\tfrac12Z_2''(0)-\mathscr R.
\end{aligned}
\]

\begin{proposition}
\label{prop:prime-weight-dirichlet-continuation}
The function $D$ has a meromorphic continuation to $\sigma>-1$.
Its only poles are at $0$, $1$ and $2$, with Laurent expansions
\begin{align*}
D(z)&=\frac{c_2}{z-2}+O(1) &&\text{as }z\to2,\\
D(z)&=-\frac{c_2}{(z-1)^2}+\frac{c_1-c_2}{z-1}+O(1) &&\text{as }z\to1,\\
D(z)&=\frac{c_2}{z^2}-\frac{c_1}{z}+d_0+d_1z+O(z^2) &&\text{as }z\to0.
\end{align*}
\end{proposition}

\begin{proof}
For $(a,b)\ne(0,0)$, both $\lambda_{a,b}$ and $u_{a,b}$ are positive,
and $\lambda_{a,b}/u_{a,b}\to1$ by~\eqref{eq:prime-weight-real-expansion}.
Hence ${\lambda_{a,b}\asymp u_{a,b}}$, and by comparison with $Z$, the
series defining $D$ converges absolutely and locally uniformly in the
half-plane ${\sigma>2}$. There, we sum \eqref{eq:prime-weight-complex-expansion}
over all $(a,b)\ne(0,0)$. Differentiating $Z$ term by term gives
\[
\sum_{(a,b)\ne(0,0)}(\log u_{a,b})^ku_{a,b}^{-w}=(-1)^kZ^{(k)}(w)
\qquad\text{for }k\ge0\text{ and }\Re w>2,
\]
and we obtain
\begin{equation}
\label{eq:prime-weight-dirichlet-decomposition}
D(z)=(\log2)^{-z}+Z(z)+zZ'(z+1)
+\frac{z(z+1)}2Z''(z+2)+zZ'(z+2)+zZ(z+2)+\mathcal H(z).
\end{equation}
By \cref{lem:lattice-zeta-continuation} and the holomorphy of
$\mathcal H$, the right-hand side of
\eqref{eq:prime-weight-dirichlet-decomposition} continues $D$
meromorphically to the half-plane $\sigma>-1$. Its only possible poles
there are $0$, $1$ and $2$.

Putting $z=0$ in \eqref{eq:prime-weight-complex-expansion}, we get
$1=1+R_{a,b}(0)$, so $R_{a,b}(0)=0$.
By~\eqref{eq:prime-weight-real-expansion} and Taylor's formula for
$\log(1+x)$ at $x=0$,
\[
\log\frac{\lambda_{a,b}}{u_{a,b}}=\frac{\log u_{a,b}}{u_{a,b}}
+\frac{\log u_{a,b}-1-\frac12(\log u_{a,b})^2}{u_{a,b}^2}
+O\!\left(\frac{(\log u_{a,b})^3}{u_{a,b}^3}\right)
\qquad\text{as }u_{a,b}\to\infty.
\]
Differentiating \eqref{eq:prime-weight-complex-expansion} with respect
to $z$ at $z=0$ shows that $-R'_{a,b}(0)$ is the term of $\mathscr R$
indexed by $(a,b)$. Using the bounds for $u_{a,b}$ in
\secref{subsec:prime-weights} and grouping the pairs by $m=a+b$, we
obtain
\[
\sum_{\substack{(a,b)\ne(0,0)\\u_{a,b}>T}}|R'_{a,b}(0)|
\ll\sum_{m>T/\max\{\log r,\log s\}}\frac{(m+1)(1+\log m)^3}{m^3}
\ll\frac{(\log T)^3}{T}
\qquad\text{as }T\to\infty.
\]
Thus $\mathscr R$ converges absolutely. The series defining $\mathcal H$
converges locally uniformly in $\Re z>-1$, so it may be differentiated
term by term. Hence $\mathcal H(0)=0$ and $\mathcal H'(0)=-\mathscr R$.

Near $z=2$, the only singular term
in~\eqref{eq:prime-weight-dirichlet-decomposition} is
$Z(z)=c_2/(z-2)+O(1)$. Near $z=1$, the singular terms are
\[
Z(z)=\frac{c_1}{z-1}+O(1)
\qquad\text{and}\qquad
zZ'(z+1)=-\frac{c_2}{(z-1)^2}-\frac{c_2}{z-1}+O(1).
\]
This proves the expansions at $2$ and $1$.
Near $z=0$, we have $zZ'(1+z)=-c_1/z+zZ_1'(z)$ and
\[
\begin{aligned}
\frac{z(z+1)}2Z''(2+z)+zZ'(2+z)+zZ(2+z)
={}&\frac{c_2}{z^2}+c_2\\
&+\frac{z(z+1)}2Z_2''(z)+zZ_2'(z)+zZ_2(z).
\end{aligned}
\]
We substitute these identities into
\eqref{eq:prime-weight-dirichlet-decomposition} and use
\[
Z(0)=c_0,
\qquad
(\log2)^{-z}=1-z\log(\log2)+O(z^2)
\qquad\text{and}\qquad
\mathcal H(z)=-\mathscr Rz+O(z^2).
\]
The constant and linear coefficients are then $d_0$ and $d_1$,
respectively. Since $c_2>0$, all three poles are nonremovable.
\end{proof}

\begin{corollary}
\label{cor:dirichlet-growth}
Let $-1<\sigma_-<\sigma_+$. For $\sigma_-\le\sigma\le\sigma_+$ and
$|t|\ge1$, we have
\[
|D(\sigma+it)|
\ll_{\sigma_-,\sigma_+}(1+|t|)^6.
\]
\end{corollary}

\begin{proof}
On the given strip, $1+|z|\ll1+|t|$.
By \cref{lem:lattice-zeta-growth}, each shifted derivative of $Z$
in \eqref{eq:prime-weight-dirichlet-decomposition} is
$O((1+|t|)^4)$.
Its polynomial coefficient has degree at most two, so these terms
are $O((1+|t|)^6)$.
The remaining terms satisfy
$(\log2)^{-z}=O(1)$ and
$\mathcal H(z)=O((1+|t|)^3)$, which proves the claim.
\end{proof}

\section{The Mellin expansion}
\label{sec:mellin}

This section expands the generating series $\sum_{n\in S}n^{-\tau}$ as
$\tau\to0$ in a sector about the positive real axis. Since the primes
$p_{r^a s^b}$ are distinct, unique factorisation gives, for every $t\ge0$,
\[
C(e^t)=\#\biggl\{(m_{a,b})_{a,b\in\mathbb N}:
m_{a,b}\in\mathbb N,\ \sum_{a,b\in\mathbb N}m_{a,b}\lambda_{a,b}\le t\biggr\}.
\]
Since $\lambda_{a,b}\ge\log2$, each such array satisfies
$\sum_{a,b}m_{a,b}\le t/\log2$ and hence has finite support. For $\tau\in\mathbb C$ with $\Re \tau>0$, set
\[
\mathcal F(\tau)=\sum_{n\in S}n^{-\tau}
\qquad\text{and}\qquad
L(\tau)=\sum_{a,b\ge0}\sum_{m\ge1}\frac{e^{-m\tau\lambda_{a,b}}}{m}.
\]
For $\Re w>0$, we let $\log w$ denote the principal logarithm and put
$w^{-z}=\exp(-z\log w)$.
By the Cahen--Mellin integral, we have, for $\Re x>0$,
\begin{equation}
\label{eq:mellin-inversion}
e^{-x}=\frac1{2\pi i}\int_{3-i\infty}^{3+i\infty}\Gamma(z)x^{-z}\,\mathrm{d}z.
\end{equation}
Applying \eqref{eq:mellin-inversion} with $x=m\tau\lambda_{a,b}$ to each
term of $L$ and summing over $m\ge1$ and $a,b\ge0$, we write $L(\tau)$ as
a Mellin integral with kernel $\Gamma(z)\zeta(z+1)D(z)$.
Moving the line of integration from $\Re z=3$ to the left of the
imaginary axis picks up the residues at $z=2,1,0$. They give the
following terms of $L(\tau)$:
\begin{center}
\renewcommand{\arraystretch}{1.25}
\setlength{\abovetopsep}{3pt}
\setlength{\belowbottomsep}{3pt}
\begin{tabular}{@{}c@{\qquad}c@{\qquad}l@{}}
\toprule
Pole & Order & Contribution to $L(\tau)$\\
\midrule
$z=2$ & $1$ & $A_1\tau^{-2}$\\
$z=1$ & $2$ & $\bigl(A_3-A_2\log(1/\tau)\bigr)\tau^{-1}$\\
$z=0$ & $4$ & a cubic polynomial in $\log(1/\tau)$\\
\bottomrule
\end{tabular}
\end{center}
The shift of the saddle point from its first approximation $W^{-1}$ is
computed in \secref{sec:saddle}. Besides the residues, it
contributes to the coefficient of $(\log W)^2$, to the power of $W$ and
to the constant $K$ in \cref{thm:prime-index-strong-asymptotic}.

For the residue at $z=0$, we use the Laurent expansion
\[
\Gamma(z)=z^{-1}-\gamma
+\left(\frac{\gamma^2}{2}+\frac{\pi^2}{12}\right)z
-\left(\frac{\gamma^3}{6}+\frac{\gamma\pi^2}{12}+\frac{\zeta(3)}3\right)z^2
+O(z^3)
\qquad\text{as }z\to0.
\]
We also use the expansion of $\zeta(1+z)$ from
\secref{sec:intro}, extended by one term,
\[
\zeta(1+z)
=\frac1z+\gamma-\gamma_1z+\frac{\gamma_2}{2}z^2+O(z^3)
\qquad\text{as }z\to0,
\]
where $\gamma_2$ is the second Stieltjes constant.
In the product of these two expansions, the terms in $z^{-1}$ cancel,
and we obtain
\begin{equation}
\label{eq:gamma-zeta-expansion}
\Gamma(z)\zeta(z+1)=z^{-2}+\kappa_0+\kappa_1z+O(z^2),
\end{equation}
where
\[
\kappa_0=\frac{\pi^2}{12}-\gamma_1-\frac{\gamma^2}{2}
\qquad\text{and}\qquad
\kappa_1=\frac{\gamma_2}{2}+\gamma\gamma_1+\frac{\gamma^3-\zeta(3)}{3}.
\]
We collect the residue contributions in the cubic polynomial $Q$ and
the function $\Phi$ defined by
\begin{equation}
\label{eq:q-phi}
\begin{aligned}
Q(X)&=\frac{c_2}{6}X^3-\frac{c_1}{2}X^2+(d_0+c_2\kappa_0)X
+d_1-c_1\kappa_0+c_2\kappa_1,\\
\Phi(\tau)&=\frac{A_1}{\tau^2}+\frac{A_3-A_2\log(1/\tau)}{\tau}
+Q(\log(1/\tau)).
\end{aligned}
\end{equation}

\begin{proposition}
\label{prop:mellin-expansion}
Let $0<\varepsilon<\pi/2$ and $0<\delta<1$. Then, uniformly for
$|\arg \tau|\le\pi/2-\varepsilon$,
\[
\mathcal F(\tau)=e^{\Phi(\tau)}\bigl(1+O_{\varepsilon,\delta}(|\tau|^\delta)\bigr)
\quad\text{and}\quad
L(\tau)=\Phi(\tau)+O_{\varepsilon,\delta}(|\tau|^\delta)
\qquad\text{as }\tau\to0.
\]
\end{proposition}

\begin{proof}
Since $p_n>n$ and $r,s\ge2$, we have $\lambda_{a,b}>(a+b)\log2$.
Together with $\lambda_{a,b}\ge\log2$, this gives
\[
\lambda_{a,b}\ge\frac{\log2}{2}(1+a+b).
\]
Hence $\sum_{a,b\ge0}e^{-\sigma_0\lambda_{a,b}}<\infty$
for every $\sigma_0>0$. For $\Re\tau\ge\sigma_0$,
\[
\sum_{a,b\ge0}\sum_{m\ge1}
\frac{|e^{-m\tau\lambda_{a,b}}|}{m}
\le
\frac1{1-2^{-\sigma_0}}
\sum_{a,b\ge0}e^{-\sigma_0\lambda_{a,b}}.
\]
Thus the series defining $L$ converges absolutely and uniformly on
$\Re\tau\ge\sigma_0$, and hence locally uniformly in $\Re\tau>0$.
Since the primes $p_{r^as^b}$ are distinct, unique factorisation gives,
for $\Re\tau>0$, the Euler product
\[
\mathcal F(\tau)
=\prod_{a,b\ge0}(1-e^{-\tau\lambda_{a,b}})^{-1}.
\]
The product converges absolutely by the preceding bound. Taking
logarithms, we obtain $\mathcal F(\tau)=e^{L(\tau)}$, since
${-\log(1-w)=\sum_{m\ge1}w^m/m}$ for $|w|<1$. Thus $L$ is a holomorphic
logarithm of $\mathcal F$, real on the positive axis.

To apply \eqref{eq:mellin-inversion} term by term, we use Stirling's
estimate, uniform for $\sigma$ in a bounded interval:
\[
|\Gamma(\sigma+it)|
\sim\sqrt{2\pi}\,|t|^{\sigma-1/2}e^{-\pi|t|/2}
\qquad\text{as }|t|\to\infty.
\]
For $z=3+it$, we have
\[
|(m\tau\lambda_{a,b})^{-z}|\le(m|\tau|\lambda_{a,b})^{-3}e^{|t\arg\tau|}.
\]
Since $|\arg\tau|<\pi/2$, Stirling's estimate therefore gives
\[
\begin{aligned}
&\sum_{a,b\ge0}\sum_{m\ge1}\frac1m
\int_{-\infty}^{\infty}
\bigl|\Gamma(3+it)
(m\tau\lambda_{a,b})^{-3-it}\bigr|\,\mathrm dt\\
&\qquad\le
|\tau|^{-3}\zeta(4)D(3)
\int_{-\infty}^{\infty}
|\Gamma(3+it)|e^{|t\arg\tau|}\,\mathrm dt
<\infty.
\end{aligned}
\]
By Fubini's theorem, we may therefore interchange summation and
integration. For $\Re z=3$, we have
\[
\sum_{m\ge1}m^{-1-z}=\zeta(z+1)
\qquad\text{and}\qquad
\sum_{a,b\ge0}\lambda_{a,b}^{-z}=D(z).
\]
Hence
\[
L(\tau)=\frac1{2\pi i}\int_{3-i\infty}^{3+i\infty}
\Gamma(z)\zeta(z+1)D(z)\tau^{-z}\,\mathrm{d}z.
\]

We now move the line of integration to $\Re z=-\delta$. In the strip
$-\delta\le\Re z\le3$, the integrand has poles only at $z=0,1,2$, and we
first bound it for $|t|\ge1$. By the bound for $\zeta$ in the proof of
\cref{lem:lattice-zeta-growth}, we have $\zeta(z+1)=O((1+|t|)^3)$ on
this strip for $|t|\ge1$. Combining this with
\cref{cor:dirichlet-growth} and Stirling's estimate, we obtain, for
$-\delta\le\sigma\le3$ and $|t|\ge1$,
\[
|\Gamma(\sigma+it)\zeta(\sigma+1+it)D(\sigma+it)|
\ll_\delta(1+|t|)^M e^{-\pi|t|/2}
\qquad\text{with }M=\tfrac{23}{2}.
\]
In the sector $|\arg\tau|\le\pi/2-\varepsilon$, we have
$|\tau^{-z}|=|\tau|^{-\sigma}e^{t\arg\tau}\le|\tau|^{-\sigma}e^{(\pi/2-\varepsilon)|t|}$,
and hence
\[
|\Gamma(z)\zeta(z+1)D(z)\tau^{-z}|
\ll_\delta|\tau|^{-\sigma}(1+|t|)^Me^{-\varepsilon|t|}.
\]
For fixed $\tau$ in this sector, the integrals over the horizontal sides
of the rectangle with vertices $3\pm iY$ and $-\delta\pm iY$ therefore
tend to zero as $Y\to\infty$. The integrand is continuous on the line
$\Re z=-\delta$, so the preceding bound also holds there for $|t|<1$.
Hence the integral over this line is
\[
O_{\varepsilon,\delta}\left(
|\tau|^\delta\int_{-\infty}^{\infty}(1+|t|)^M e^{-\varepsilon|t|}\,\mathrm{d}t
\right)=O_{\varepsilon,\delta}(|\tau|^\delta).
\]

Applying the residue theorem to the truncated contours and letting $Y\to\infty$
expresses $L(\tau)$ as the sum of
the residues at $z=2,1,0$, with error
$O_{\varepsilon,\delta}(|\tau|^\delta)$.

At the simple pole $z=2$, \cref{prop:prime-weight-dirichlet-continuation}
and $\Gamma(2)=1$ give the residue
\[
c_2\zeta(3)\tau^{-2}=\frac{A_1}{\tau^2}.
\]
At the double pole $z=1$, the identities $\Gamma(1)=1$ and
$\Gamma'(1)=-\gamma$ give
\[
\left.
\frac{\mathrm d}{\mathrm dz}
\bigl(\Gamma(z)\zeta(z+1)\bigr)
\right|_{z=1}
=\zeta'(2)-\gamma\zeta(2).
\]
The expansion of $D$ at $z=1$ in
\cref{prop:prime-weight-dirichlet-continuation} then gives the residue
\[
\frac{
(c_1-c_2)\zeta(2)-c_2\zeta'(2)
+c_2\zeta(2)(\gamma+\log\tau)
}{\tau}
=
\frac{A_3-A_2\log(1/\tau)}{\tau}.
\]
At zero, multiplying~\eqref{eq:gamma-zeta-expansion} by the expansion of
$D$ in \cref{prop:prime-weight-dirichlet-continuation} gives
\[
\Gamma(z)\zeta(z+1)D(z)
=\frac{c_2}{z^4}-\frac{c_1}{z^3}
+\frac{d_0+c_2\kappa_0}{z^2}
+\frac{Q(0)}z+O(1).
\]
Expanding $\tau^{-z}=e^{z\log(1/\tau)}$ in powers of $z$, we find that
the residue is $Q(\log(1/\tau))$.
The three residues sum to $\Phi(\tau)$. Hence
${L(\tau)=\Phi(\tau)+O_{\varepsilon,\delta}(|\tau|^\delta)}$.
The formula for $\mathcal F$ follows from $\mathcal F=e^L$.
\end{proof}

\begin{remark*}
On the positive real axis, the first two terms of $\Phi$
in~\eqref{eq:q-phi} can be recovered directly from
\cref{cor:geo-prime-count}. For $\tau>0$, let
\[
\Lambda(\tau)=\sum_{a,b\ge0}e^{-\tau\lambda_{a,b}},
\qquad\text{so that}\qquad
L(\tau)=\sum_{m\ge1}\frac{\Lambda(m\tau)}m.
\]
Integration by parts gives
\[
\Lambda(\tau)=\tau\int_0^\infty e^{-\tau v}N_\lambda(v)\,\mathrm{d}v.
\]
In terms of $c_1$ and $c_2$, \cref{cor:geo-prime-count} reads
\[
N_\lambda(v)=\frac{c_2}2v^2-c_2v\log v+c_1v+E(v)
\qquad\text{with }E(v)=o(v)\text{ as }v\to\infty.
\]
Since $N_\lambda$ is nondecreasing and $v\log v\to0$ as $v\to0^+$, we
also have $E(v)=O(1+v)$ for $v>0$. Let
$G(\tau)=\tau^2\int_0^\infty e^{-\tau v}E(v)\,\mathrm{d}v$. Substituting
$v=w/\tau$ and using $\Gamma'(2)=1-\gamma$, we obtain
\[
\tau\Lambda(\tau)=\frac{c_2}{\tau}+c_2\log\tau+c_1+c_2(\gamma-1)+G(\tau)
\qquad\text{for }\tau>0.
\]
Since $G(\tau)=\int_0^\infty e^{-w}\,\tau E(w/\tau)\,\mathrm{d}w$, the
bounds for $E$ and dominated convergence give $G(\tau)\to0$ as
$\tau\to0^+$. The same bounds give $G(\tau)=O(1)$ for $0<\tau\le1$. For
$\tau\ge1$, the bound $\lambda_{a,b}\ge\log2$ gives
$\Lambda(\tau)\le2^{1-\tau}\Lambda(1)$, so $\tau\Lambda(\tau)$ is bounded.
The expansion of $\tau\Lambda(\tau)$ then gives $G(\tau)=O(1+\log\tau)$
for $\tau\ge1$. Together with the bound for $0<\tau\le1$, this yields
$G(\tau)\ll1+\log(1+\tau)$ for all $\tau>0$. In particular, for
$0<\tau\le1$ and $m\ge1$, we have
\[
\frac{|G(m\tau)|}{m^2}\ll\frac{1+\log(1+m)}{m^2}.
\]
The right-hand side is summable over $m$, so dominated convergence gives
$\sum_{m\ge1}G(m\tau)/m^2\to0$ as $\tau\to0^+$. Inserting the
expansion of $\tau\Lambda(\tau)$, with $\tau$ replaced by $m\tau$, into the
series for $L$ and using
$\sum_{m\ge1}m^{-2}\log m=-\zeta'(2)$, we obtain
\[
L(\tau)=\frac{A_1}{\tau^2}+\frac{A_3-A_2\log(1/\tau)}{\tau}+o(\tau^{-1})
\qquad\text{as }\tau\to0^+.
\]
Thus the logarithmic term $-A_2\log(1/\tau)/\tau$ arises from the summand
$-c_2v\log v$ in $N_\lambda(v)$. Since \cref{cor:geo-prime-count} has error term $o(v)$, this argument
determines $L(\tau)$ only up to $o(\tau^{-1})$. It therefore cannot detect
the term $Q(\log(1/\tau))$ of $\Phi$.
\end{remark*}

\section{Tauberian estimates and the saddle point}
\label{sec:saddle}

To pass from the expansion of $\mathcal F$ in
\cref{prop:mellin-expansion} to the asymptotics of $C$, we use the
following form of Ingham's Tauberian theorem.

\begin{samepage}
\begin{theorem}[Ingham, {\cite[Thm.~4.1]{BringmannJenningsShafferMahlburg2023}}]
\label{thm:ingham}
Let $N:[0,\infty)\to\mathbb R$ be nondecreasing and right-continuous,
with $N(0)=0$, and assume that its Laplace--Stieltjes transform
\[
f(\tau)=\int_{[0,\infty)}e^{-\tau t}\,\mathrm{d}N(t)
\]
converges for $\Re\tau>0$. Let $h>0$, and let $U$ be a connected open
subset of the half-plane $\Re\tau>0$ that contains $(0,h]$. For
$0<\sigma\le h$, let $\rho(\sigma)$ denote the distance from $\sigma$ to
$\mathbb C\setminus U$. Let $\varphi$ be holomorphic on $U$, with
$\varphi(\sigma)>0$ and $\varphi''(\sigma)>0$ for $0<\sigma\le h$. Assume
that the following hold as $\sigma\to0^+$:
\begin{enumerate}
\item $-\sigma\varphi'(\sigma)\to\infty$,
\item $-\rho(\sigma)\varphi'(\sigma)/\bigl(\sigma\sqrt{\varphi''(\sigma)}\bigr)\to\infty$,
\item $|\varphi''(\sigma+w)|\ll\varphi''(\sigma)$ uniformly for $|w|<\rho(\sigma)$.
\end{enumerate}
Assume further that $f(\tau)\sim e^{\varphi(\tau)}$ as $\tau\to0$ in $U$,
and that, for each $\Delta>0$,
\[
|f(\tau)|\ll_\Delta e^{\varphi(|\tau|)}
\qquad\text{as }\tau\to0\text{ with }|\Im\tau|\le\Delta\Re\tau.
\]
Then, for all sufficiently large $t$, the equation $-\varphi'(\tau)=t$ has
a unique solution $\tau_t\in(0,h]$, and
\[
N(t)\sim
\frac{\exp\bigl(\varphi(\tau_t)+t\tau_t\bigr)}
{\tau_t\sqrt{2\pi\varphi''(\tau_t)}}
\qquad\text{as }t\to\infty.
\]
\end{theorem}
\end{samepage}

We apply \cref{thm:ingham} to $N(t)=C(e^t)-1$, with $\varphi=\Phi$.

\begin{lemma}
\label{lem:tauberian-admissibility}
There exists $\sigma_*>0$ such that, for all large $t$, the equation
$-\Phi'(\tau)=t$ has a unique solution $\tau_t\in(0,\sigma_*)$, and
\[
C(e^t)\sim
\frac{\exp\bigl(\Phi(\tau_t)+t\tau_t\bigr)}
{\tau_t\sqrt{2\pi\Phi''(\tau_t)}}
\qquad\text{as }t\to\infty.
\]
\end{lemma}

\begin{proof}
Let $N(t)=C(e^t)-1$ for $t\ge0$. Since
$C(x)=\#\{n\in S:n\le x\}$ and $C(1)=1$, the function $N$ is
nondecreasing and right-continuous, with $N(0)=0$.
Thus $\mathrm{d}N$ assigns unit mass to each point $\log n$ with
$n\in S\setminus\{1\}$. By the proof of \cref{prop:mellin-expansion}, the
series defining $\mathcal F$ converges absolutely for $\Re\tau>0$. Hence
the Laplace--Stieltjes transform of $N$ converges there, and
\[
f(\tau)=\int_{[0,\infty)}e^{-\tau t}\,\mathrm{d}N(t)
=\sum_{n\in S\setminus\{1\}}n^{-\tau}=\mathcal F(\tau)-1.
\]

Differentiation gives
\begin{equation}
\label{eq:tauberian-proof-derivatives}
\begin{aligned}
\Phi'(\tau)&=-\frac{2A_1}{\tau^3}
+\frac{A_2-A_3+A_2\log(1/\tau)}{\tau^2}-\frac{Q'(\log(1/\tau))}{\tau},\\
\Phi''(\tau)&=\frac{6A_1}{\tau^4}
+\frac{-2A_2\log(1/\tau)+2A_3-3A_2}{\tau^3}
+\frac{Q'(\log(1/\tau))+Q''(\log(1/\tau))}{\tau^2}.
\end{aligned}
\end{equation}
Since $A_1>0$, the definition of $\Phi$
and~\eqref{eq:tauberian-proof-derivatives} give
\[
\Phi(\sigma)\sim\frac{A_1}{\sigma^2},\qquad
-\Phi'(\sigma)\sim\frac{2A_1}{\sigma^3}
\qquad\text{and}\qquad
\Phi''(\sigma)\sim\frac{6A_1}{\sigma^4}
\qquad\text{as }\sigma\to0^+.
\]
Hence there exists $\sigma_*>0$ such that $\Phi(\sigma)$, $-\Phi'(\sigma)$
and $\Phi''(\sigma)$ are positive for $0<\sigma\le\sigma_*$. On this interval, $-\Phi'$ is strictly decreasing, and
it tends to infinity as $\sigma\to0^+$. Thus, for
$t>-\Phi'(\sigma_*)$, the equation $-\Phi'(\tau)=t$ has a unique solution
$\tau_t$ with $0<\tau_t<\sigma_*$. Moreover, $\tau_t\to0$ as $t\to\infty$,
and $-\Phi'(\sigma)\sim2A_1/\sigma^3$ gives $\tau_t\sim(2A_1/t)^{1/3}$.

We first fix the region $U$ of \cref{thm:ingham} and verify that
$f(\tau)\sim e^{\Phi(\tau)}$ as $\tau\to0$ in $U$.
Let $U$ be the open sector $|\arg\tau|<\theta_0$, where
$0<\theta_0<\pi/4$ is fixed, so that $\cos(2\theta_0)>0$.
Since $U$ lies in the half-plane $\Re\tau>0$, where $\log(1/\tau)$ is
holomorphic, the function $\Phi$ is holomorphic on $U$.
For $\sigma>0$, the boundary rays $\arg\tau=\pm\theta_0$ of $U$ are at
distance $\sigma\sin\theta_0$ from $\sigma$, so $\rho(\sigma)=\sigma\sin\theta_0$.
We use polar coordinates $\tau=\xi e^{i\vartheta}$ and let
$\ell=\log(1/\xi)$. Then, uniformly
for $|\vartheta|<\theta_0$,
\[
\Re\Phi(\tau)\ge\frac{A_1\cos(2\theta_0)}{\xi^2}
-O_{\theta_0}\!\left(\frac{1+\ell}{\xi}+(1+\ell)^3\right)
\gg_{\theta_0}\xi^{-2}
\qquad\text{as }\xi\to0^+.
\]
Hence $e^{-\Phi(\tau)}=O_{\theta_0}(|\tau|^{\delta})$ for every $\delta>0$,
uniformly in $U$. By \cref{prop:mellin-expansion}, for every
$0<\delta<1$, uniformly in $U$,
\[
f(\tau)=\mathcal F(\tau)-1
=e^{\Phi(\tau)}\bigl(1+O_{\theta_0,\delta}(|\tau|^\delta)\bigr)
\qquad\text{as }\tau\to0.
\]

We next verify the bound $|f(\tau)|\ll_\Delta e^{\Phi(|\tau|)}$ of
\cref{thm:ingham} in the sector $|\Im\tau|\le\Delta\Re\tau$, for fixed
$\Delta>0$.
A calculation gives
\[
\Re\Phi(\xi e^{i\vartheta})-\Phi(\xi)
=-\frac{2A_1\sin^2\vartheta}{\xi^2}
+\frac{(A_3-A_2\ell)(\cos\vartheta-1)+A_2\vartheta\sin\vartheta}{\xi}
+\frac{c_1-c_2\ell}{2}\vartheta^2.
\]
For $|\vartheta|\le\arctan\Delta$, we have
$|\cos\vartheta-1|\le\vartheta^2/2$, $|\vartheta\sin\vartheta|\le\vartheta^2$
and $\vartheta^2\ll_\Delta\sin^2\vartheta$.
Hence, for $0<\xi<1$, the difference
$\Re\Phi(\xi e^{i\vartheta})-\Phi(\xi)$ is at most
\[
\left[-\frac{2A_1}{\xi^2}
+O_\Delta\!\left(\frac{1+\ell}{\xi}+1+\ell\right)\right]\sin^2\vartheta.
\]
Since $\xi(1+\ell)\to0$ as $\xi\to0^+$, the bracket is negative for all
sufficiently small $\xi$. Hence $\Re\Phi(\tau)\le\Phi(|\tau|)$ for all
sufficiently small $\tau$ with $|\Im\tau|\le\Delta\Re\tau$.
Since $\arctan\Delta<\pi/2$, \cref{prop:mellin-expansion} applies in this
sector and gives $|\mathcal F(\tau)|\ll_\Delta e^{\Re\Phi(\tau)}$. Hence
\[
|f(\tau)|\le|\mathcal F(\tau)|+1\ll_\Delta e^{\Phi(|\tau|)}
\qquad\text{as }\tau\to0\text{ with }|\Im\tau|\le\Delta\Re\tau.
\]

It remains to verify conditions~(i)--(iii) of \cref{thm:ingham}. Since
$\rho(\sigma)=\sigma\sin\theta_0$, the asymptotics of $\Phi'$ and $\Phi''$
above give
\[
-\sigma\Phi'(\sigma)\sim\frac{2A_1}{\sigma^2}
\qquad\text{and}\qquad
-\frac{\rho(\sigma)\Phi'(\sigma)}{\sigma\sqrt{\Phi''(\sigma)}}
\sim\sin\theta_0\sqrt{\frac{2A_1}{3}}\,\frac1\sigma
\qquad\text{as }\sigma\to0^+.
\]
Both right-hand sides tend to infinity, so conditions~(i) and~(ii) hold.
For condition~(iii), let $0<\sigma<1$ and $|w|<\rho(\sigma)$. Then
$\sigma+w$ lies in $U$, with $|\sigma+w|\asymp_{\theta_0}\sigma$ and
$|\log(\sigma+w)|\ll_{\theta_0}1+\log(1/\sigma)$.
Inserting these bounds into~\eqref{eq:tauberian-proof-derivatives} and
using $\Phi''(\sigma)\sim6A_1/\sigma^4$, we obtain
\[
|\Phi''(\sigma+w)|\ll_{\theta_0}\sigma^{-4}\ll\Phi''(\sigma).
\]
This is condition~(iii).
Thus \cref{thm:ingham} applies on $U$, with $\varphi=\Phi$ and
$h=\sigma_*$. It gives the asymptotic formula of the lemma for $N(t)$ in
place of $C(e^t)$.
Since $2=p_{r^0s^0}$ lies in $\mathcal P$, the integers $2^k$ with
$1\le k\le t/\log2$ belong to $S$. Hence $N(t)\ge\lfloor t/\log2\rfloor$,
and in particular $N(t)\to\infty$.
Therefore $C(e^t)=N(t)+1\sim N(t)$, and the lemma follows.
\end{proof}

Recall that $\tau_t$ denotes the unique solution of $-\Phi'(\tau)=t$ with
$0<\tau<\sigma_*$, defined for all large $t$. The next lemma expresses the
exponent $\Phi(\tau_t)+t\tau_t$ and the factor $\tau_t\sqrt{2\pi\Phi''(\tau_t)}$
from \cref{lem:tauberian-admissibility} explicitly in terms of $t$.

\begin{samepage}
\begin{lemma}
\label{lem:explicit-saddle-expansion}
Let $W=(t/(2A_1))^{1/3}$, so that $t=2A_1W^3$. Then, as $t\to\infty$,
\begin{enumerate}
\item $\displaystyle\Phi(\tau_t)+t\tau_t
=3A_1W^2+(A_3-A_2\log W)W+Q(\log W)
-\frac{(A_3-A_2-A_2\log W)^2}{12A_1}
+O\!\left(\frac{(\log W)^3}{W}\right)$;
\item $\displaystyle\tau_t\sqrt{2\pi\Phi''(\tau_t)}
=\sqrt{12\pi A_1}\,W\left(1+O\!\left(\frac{\log W}{W}\right)\right)$.
\end{enumerate}
\end{lemma}
\end{samepage}

\begin{proof}
Let $B_W=A_3-A_2-A_2\log W$. Since $Q$ is a cubic polynomial, we have
$B_W\ll\log W$ and $Q'(\log W)\ll(\log W)^2$ for large $W$. By the proof
of \cref{lem:tauberian-admissibility}, we have $\tau_t\sim W^{-1}$.
Evaluating~\eqref{eq:tauberian-proof-derivatives} at $\tau=W^{-1}$ and
using $t=2A_1W^3$, we obtain
\[
-\Phi'(W^{-1})-t=B_WW^2+Q'(\log W)W.
\]
Every $\tau$ in the closed interval with endpoints $W^{-1}$ and $\tau_t$
satisfies $\tau\sim W^{-1}$. Hence~\eqref{eq:tauberian-proof-derivatives}
gives $\Phi''(\tau)\asymp W^4$ on this interval.
Since $-\Phi'(\tau_t)=t$, the mean-value theorem applied to $\Phi'$ gives
\[
|\tau_t-W^{-1}|
\ll W^{-4}\bigl|\Phi'(\tau_t)-\Phi'(W^{-1})\bigr|
=W^{-4}\bigl|B_WW^2+Q'(\log W)W\bigr|
\ll W^{-2}\log W.
\]
Thus every $\tau$ in this interval satisfies
$\tau=W^{-1}\bigl(1+O(W^{-1}\log W)\bigr)$. Inserting this
into~\eqref{eq:tauberian-proof-derivatives}, we obtain
\[
\Phi''(\tau)
=6A_1W^4
\left(1+O\!\left(\frac{\log W}{W}\right)\right)
\qquad\text{for all such }\tau.
\]
Applying the mean-value theorem once more, we find some $\eta$ in this
interval such that
\[
\tau_t-W^{-1}=\frac{B_WW^2+Q'(\log W)W}{\Phi''(\eta)}.
\]
The estimate for $\Phi''$ above then gives
\[
\tau_t
=W^{-1}+\frac{B_W}{6A_1W^2}
+O\!\left(\frac{(\log W)^2}{W^3}\right).
\]
By the definition~\eqref{eq:q-phi} of $\Phi$ and $t=2A_1W^3$, we have
\[
\Phi(W^{-1})+tW^{-1}
=3A_1W^2+(A_3-A_2\log W)W+Q(\log W).
\]
Since $\Phi'(W^{-1})+t=-B_WW^2-Q'(\log W)W$, Taylor's theorem for
$\Phi(\tau)+t\tau$ at $W^{-1}$ and the estimate for $\Phi''$ above give
\[
\Phi(\tau_t)+t\tau_t
=\Phi(W^{-1})+tW^{-1}-B_WW^2(\tau_t-W^{-1})
+3A_1W^4(\tau_t-W^{-1})^2+O\!\left(\frac{(\log W)^3}{W}\right).
\]
The error term contains $-Q'(\log W)W(\tau_t-W^{-1})$ and a term
$O\bigl(W^3(\tau_t-W^{-1})^2\log W\bigr)$ from the estimate for $\Phi''$.
Both are $O((\log W)^3/W)$, since $\tau_t-W^{-1}\ll W^{-2}\log W$.

Completing the square in $\tau_t-W^{-1}$, we obtain
\[
3A_1W^4(\tau_t-W^{-1})^2-B_WW^2(\tau_t-W^{-1})
=-\frac{B_W^2}{12A_1}
+3A_1W^4\left(\tau_t-W^{-1}-\frac{B_W}{6A_1W^2}\right)^2.
\]
By the expansion of $\tau_t$, the last term is
$O((\log W)^4/W^2)=o((\log W)^3/W)$. This proves~(i).
For part~(ii), the expansion of $\tau_t$ gives
$\tau_t=W^{-1}(1+O(\log W/W))$, and the estimate for $\Phi''$ gives
${\Phi''(\tau_t)=6A_1W^4(1+O(\log W/W))}$. Since
$\sqrt{2\pi\cdot6A_1}=\sqrt{12\pi A_1}$, part~(ii) follows.
\end{proof}

\section{Proof of Theorem~\ref{thm:prime-index-strong-asymptotic} and an
explicit formula for the constant}
\label{sec:main-proof}

The constant in \cref{thm:prime-index-strong-asymptotic} is
\begin{equation}
\label{eq:constant-Q}
K=\frac1{\sqrt{12\pi A_1}}
\exp\!\left(Q(0)-\frac{(A_3-A_2)^2}{12A_1}\right).
\end{equation}
Since $Q(0)=d_1-c_1\kappa_0+c_2\kappa_1$ by~\eqref{eq:q-phi}, the
definition of $d_1$ gives the explicit expression
\begin{equation}
\label{eq:constant}
\begin{aligned}
K=\frac1{(\log2)\sqrt{12\pi A_1}}
\exp\bigg(&Z'(0)+Z_1'(0)+Z_2(0)+Z_2'(0)+\tfrac12Z_2''(0)\\
&-\mathscr R-c_1\kappa_0+c_2\kappa_1-\frac{(A_3-A_2)^2}{12A_1}\bigg).
\end{aligned}
\end{equation}
\begin{proof}[\upshape\bfseries Proof of \cref{thm:prime-index-strong-asymptotic}]
By \cref{lem:tauberian-admissibility,lem:explicit-saddle-expansion},
with $W$ as in \cref{thm:prime-index-strong-asymptotic}, we have
\[
C(x)\sim
\frac{\exp\!\left(3A_1W^2+(A_3-A_2\log W)W+Q(\log W)
-(A_3-A_2-A_2\log W)^2/(12A_1)\right)}
{\sqrt{12\pi A_1}\,W}.
\]
Since \cref{lem:tauberian-admissibility} is only an asymptotic
equivalence, the relative error remains $o(1)$. The definition of $\nu$
is equivalent to
\[
\nu=c_0+c_2(1+\kappa_0)+\frac{A_2(A_3-A_2)}{6A_1}.
\]
The factor $W^{-1}$ in the denominator contributes $-\log W$ to the
exponent. Expanding the square and using~\eqref{eq:q-phi} and
$d_0=1+c_0+c_2$, we obtain
\[
\begin{aligned}
Q(\log W)-\frac{(A_3-A_2-A_2\log W)^2}{12A_1}-\log W
={}&\frac{c_2}{6}(\log W)^3
-\left(\frac{c_1}{2}+\frac{A_2^2}{12A_1}\right)(\log W)^2+\nu\log W\\
&+Q(0)-\frac{(A_3-A_2)^2}{12A_1}.
\end{aligned}
\]
Since $c_2=1/(\log r\log s)$ and $c_1=\log(rs)/(2\log r\log s)$, this
proves~\eqref{eq:main-theorem-strong-asymptotic}, with $K$ given
by~\eqref{eq:constant-Q}.

It remains to check that $0<K<\infty$. The series $\mathscr R$
in~\eqref{eq:constant} converges absolutely, as shown in the proof of
\cref{prop:prime-weight-dirichlet-continuation}. By
\cref{lem:lattice-zeta-continuation}, the functions $Z$, $Z_1$ and $Z_2$
are holomorphic at zero, so the values in~\eqref{eq:constant} are finite.
They are also real, since $Z$ is real on the real axis. Finally, $A_1>0$, so
formula~\eqref{eq:constant} expresses $K$ as a positive factor times the
exponential of a real number. Hence $0<K<\infty$.
\end{proof}

A numerical evaluation of $K$ requires truncating the series
$\mathscr R$. For $T>0$, let $K_T$ be the number obtained
from~\eqref{eq:constant} by restricting the sum defining $\mathscr R$ to
$u_{a,b}\le T$. Exponentiating the tail bound for $\mathscr R$ from the
proof of \cref{prop:prime-weight-dirichlet-continuation}, we obtain
\[
\frac{K_T}{K}
=1+O\!\left(\frac{(\log T)^3}{T}\right)
\qquad\text{as }T\to\infty.
\]

\section{Concluding remarks}
\label{sec:concluding}

For $(r,s)=(2,3)$, the constants in \cref{thm:prime-index-strong-asymptotic}
are $A_1\approx1.5785$, $A_2\approx2.1601$, $A_3\approx2.2531$ and
$\nu\approx1.7260$.
The corresponding logarithmic terms in the exponent
of~\eqref{eq:main-theorem-strong-asymptotic} are approximately
$0.2189(\log W)^3$ and $-0.8346(\log W)^2$.

We do not compute $K$ numerically here. A certified value would require
an explicit implied constant in the estimate for $K_T/K$ in
\secref{sec:main-proof}.

For $k\ge3$, the prime codes of ordinals below $\omega^{\omega^k}$ are
the finite products of primes $p_n$ with indices of the form
\[
n=\prod_{j=0}^{k-1}p_{2^j}^{\,a_j}
\qquad\text{with }a_0,\ldots,a_{k-1}\in\mathbb N.
\]
The empty product is included. Indeed, the finite ordinal $j$ has code
$2^j$. Hence $\omega^j$ has code $p_{2^j}$. Every exponent below
$\omega^k$ has the form $\omega^{k-1}a_{k-1}+\cdots+\omega a_1+a_0$.
The lattice zeta function is then $k$-dimensional, with at most simple
poles at $k,k-1,\ldots,1$. Extending our method to this case would
require a higher-order expansion of the prime weights, together with
analogues of
\cref{prop:prime-weight-dirichlet-continuation,prop:mellin-expansion}.
We do not establish these here.

\makeatletter
\renewenvironment{thebibliography}[1]{%
  \@ifundefined{BR@starttoc}{}{\BR@starttoc}
  \@bibtitlestyle
  \normalfont\bibliofont\labelsep .5em\relax
  \renewcommand\theenumiv{\arabic{enumiv}}\let\p@enumiv\@empty
  \list{\@biblabel{\theenumiv}}{%
    \settowidth\labelwidth{\@biblabel{#1}}%
    \itemindent-2em
    \leftmargin\labelwidth \advance\leftmargin\labelsep
    \advance\leftmargin-\itemindent
    \renewcommand\makelabel[1]{\textnormal{##1}\hfil}%
    \usecounter{enumiv}}%
  \sloppy \clubpenalty\@M \widowpenalty\clubpenalty
  \sfcode`\.=\@m
}{%
  \def\@noitemerr{\@latex@warning{Empty `thebibliography' environment}}%
  \endlist
}
\makeatother

\hypersetup{linkcolor=backrefcolor}

\end{document}

%% file: lattice_prime_weights_figure.tex
\begin{figure}[H]
\centering
\definecolor{geonavy}{RGB}{27,58,92}
\definecolor{geogold}{RGB}{184,134,11}
\begin{tikzpicture}[x=0.7cm,y=0.7cm,font=\footnotesize,
  line cap=round,line join=round,
  inner/.style={fill=geonavy,draw=white,line width=0.5pt},
  strip/.style={fill=white,draw=geogold!80!black,line width=0.8pt},
  innerline/.style={geonavy,line width=0.7pt},
  outerline/.style={geogold!80!black,line width=0.7pt,
                    dash pattern=on 3pt off 2pt}]
 \pgfmathsetmacro{\outerX}{8/ln(2)}
 \pgfmathsetmacro{\outerY}{8/ln(3)}
 \pgfmathsetmacro{\innerX}{ln(429)/ln(2)}
 \pgfmathsetmacro{\innerY}{ln(429)/ln(3)}

 \fill[geonavy!8] (0,0)--(\innerX,0)--(0,\innerY)--cycle;
 \fill[geogold!22] (\innerX,0)--(\outerX,0)--(0,\outerY)--(0,\innerY)--cycle;
 \draw[outerline] (\outerX,0)--(0,\outerY);
 \draw[innerline] (\innerX,0)--(0,\innerY);

 \begin{scope}[line width=0.4pt,black!80]
   \draw (0,-0.7)--(12,-0.7);
   \draw (-0.7,0)--(-0.7,8);
   \foreach \x in {0,...,12} \draw (\x,-0.7)--(\x,-0.88);
   \foreach \y in {0,...,8} \draw (-0.7,\y)--(-0.88,\y);
 \end{scope}
 \foreach \x in {0,2,...,12} \node[below=1pt] at (\x,-0.88) {$\x$};
 \foreach \y in {0,2,...,8} \node[left=1pt] at (-0.88,\y) {$\y$};
 \node at (6,-2.05) {$x$};
 \node at (-2.05,4) {$y$};

 \foreach \b/\ai/\ao in {0/8/11,1/7/9,2/5/8,3/3/6,4/2/5,5/0/3,6/-1/2,7/-1/0}{
   \foreach \a in {0,...,11}{
     \ifnum\a>\ao\relax
     \else
       \ifnum\a>\ai\relax
         \path[strip] (\a,\b) circle[radius=1.75pt];
       \else
         \path[inner] (\a,\b) circle[radius=2.1pt];
       \fi
     \fi
   }
 }

\end{tikzpicture}
\caption{The case $(r,s)=(2,3)$ and $t=8$. The dashed line is
$x\log2+y\log3=8$, and the solid line is $x\log2+y\log3=\log429$. Of
the $52$ points with $u_{a,b}\le8$, the $31$ filled points satisfy
$\lambda_{a,b}\le8$, equivalently $u_{a,b}\le\log429$. The remaining
$21$ points are drawn open.}
\label{fig:geo-prime-weights}
\end{figure}

%% file: prime-index-asymptotics.bbl
\begin{thebibliography}{BBBF24}

\bibitem[AT13]{AriasDeReynaToulisse2013}
J.~Arias de Reyna and J.~Toulisse,
\emph{The $n$-th prime asymptotically},
J. Th\'eor. Nombres Bordeaux \textbf{25} (2013), 521--555.
\href{https://doi.org/10.5802/jtnb.847}{doi:10.5802/jtnb.847}.

\bibitem[BBBF24]{BridgesBrindleBringmannFranke2024}
W.~Bridges, B.~Brindle, K.~Bringmann and J.~Franke,
\emph{Asymptotic expansions for partitions generated by infinite products},
Math. Ann. \textbf{390} (2024), 2593--2632.
\href{https://doi.org/10.1007/s00208-024-02807-x}{doi:10.1007/s00208-024-02807-x}.

\bibitem[BJM23]{BringmannJenningsShafferMahlburg2023}
K.~Bringmann, C.~Jennings-Shaffer and K.~Mahlburg,
\emph{On a Tauberian theorem of Ingham and Euler--Maclaurin summation},
Ramanujan J. \textbf{61} (2023), 55--86.
\href{https://doi.org/10.1007/s11139-020-00377-5}{doi:10.1007/s11139-020-00377-5}.

\bibitem[DT20]{DebruyneTenenbaum2020}
G.~Debruyne and G.~Tenenbaum,
\emph{The saddle-point method for general partition functions},
Indag. Math. \textbf{31} (2020), 728--738.
\href{https://doi.org/10.1016/j.indag.2020.06.010}{doi:10.1016/j.indag.2020.06.010}.

\bibitem[Eve92]{Eve92}
G.~R.~Everest,
\emph{Uniform distribution and lattice point counting},
J. Austral. Math. Soc. Ser. A \textbf{53} (1992), 39--50.
\href{https://doi.org/10.1017/S1446788700035370}{doi:10.1017/S1446788700035370}.

\bibitem[G\"ob80]{Gobel1980}
F.~G\"obel,
\emph{On a 1-1-correspondence between rooted trees and natural numbers},
J. Combin. Theory Ser. B \textbf{29} (1980), 141--143.
\href{https://doi.org/10.1016/0095-8956(80)90049-0}{doi:10.1016/0095-8956(80)90049-0}.

\bibitem[GS12]{GranovskyStark2012}
B.~L.~Granovsky and D.~Stark,
\emph{A Meinardus theorem with multiple singularities},
Commun. Math. Phys. \textbf{314} (2012), 329--350.
\href{https://doi.org/10.1007/s00220-012-1526-8}{doi:10.1007/s00220-012-1526-8}.

\bibitem[HL22]{HL22}
G.~H.~Hardy and J.~E.~Littlewood,
\emph{Some problems of Diophantine approximation: The lattice-points of a
right-angled triangle},
Proc. London Math. Soc. (2) \textbf{20} (1922), 15--36.
\href{https://doi.org/10.1112/plms/s2-20.1.15}{doi:10.1112/plms/s2-20.1.15}.

\bibitem[Mat68]{Matula1968}
D.~W.~Matula,
\emph{A natural rooted tree enumeration by prime factorization},
SIAM Rev. \textbf{10} (1968), 273.

\bibitem[Ney17]{Neyt2017}
L.~Neyt,
\emph{Asymptotic distribution of integers with certain prime
factorizations},
Seminar slides, PhD Seminars, Ghent University, 14 December 2017.
\url{https://lennyneyt.com/files/slides/ADIPFII.pdf}.

\bibitem[Rui00]{Ruijsenaars2000}
S.~N.~M.~Ruijsenaars,
\emph{On Barnes' multiple zeta and gamma functions},
Adv. Math. \textbf{156} (2000), 107--132.
\href{https://doi.org/10.1006/aima.2000.1946}{doi:10.1006/aima.2000.1946}.

\bibitem[VVW14]{VVW2014}
H.~Vernaeve, J.~Vindas and A.~Weiermann,
\emph{Asymptotic distribution of integers with certain prime factorizations},
J. Number Theory \textbf{136} (2014), 87--99.
\href{https://doi.org/10.1016/j.jnt.2013.09.001}{doi:10.1016/j.jnt.2013.09.001}.

\bibitem[Wey16]{Wey16}
H.~Weyl,
\emph{{\"U}ber die Gleichverteilung von Zahlen mod.\ Eins},
Math. Ann. \textbf{77} (1916), 313--352.
\href{https://doi.org/10.1007/BF01475864}{doi:10.1007/BF01475864}.

\end{thebibliography}
